\documentclass[12pt,reqno]{amsart}
\usepackage[a4paper]{geometry}
\usepackage{amssymb,amsmath,amsthm,amsbsy,enumerate,bbm,verbatim}

\DeclareMathOperator{\Ker}{Ker}
\DeclareMathOperator{\Tr}{Tr}

\DeclareMathOperator{\supp}{supp}
\DeclareMathOperator{\sech}{sech}
\DeclareMathOperator{\arcsech}{arcsech}
\DeclareMathOperator{\dn}{dn}
\DeclareMathOperator{\cn}{cn}
\DeclareMathOperator{\essinf}{ess\, inf}
\DeclareMathOperator{\esssup}{ess\, sup}

\newcommand{\abs}[1]{\lvert#1\rvert}
\newcommand{\Abs}[1]{\left\lvert#1\right\rvert}
\newcommand{\norm}[1]{\lVert#1\rVert}

\newcommand{\jap}[1]{\langle#1\rangle}

\newcommand{\bE}{{\mathbf E}}
\newcommand{\bG}{{\mathbf G}}
\newcommand{\bL}{{\mathbf L}}
\newcommand{\bbP}{{\mathbf P}}
\newcommand{\bbR}{{\mathbb R}}
\newcommand{\bbT}{{\mathbb T}}
\newcommand{\bbZ}{{\mathbb Z}}

\newcommand{\bH}{{\mathbf H}}
\newcommand{\bh}{{\mathbf h}}

\newcommand{\calE}{\mathcal{E}}
\newcommand{\calF}{\mathcal{F}}

\newcommand{\calU}{\mathcal{U}}

\numberwithin{equation}{section}

\theoremstyle{plain}
\newtheorem{theorem}{\bf Theorem}[section]
\newtheorem*{theorem*}{Theorem}
\newtheorem{lemma}[theorem]{\bf Lemma}
\newtheorem{proposition}[theorem]{\bf Proposition}
\newtheorem*{proposition*}{\bf Proposition}

\newtheorem{corollary}[theorem]{\bf Corollary}

\theoremstyle{definition}
\newtheorem{definition}[theorem]{\bf Definition}

\theoremstyle{remark}
\newtheorem*{remark*}{\bf Remark}
\newtheorem{remark}[theorem]{\bf Remark}
\newtheorem{example}[theorem]{\bf Example}

\newcommand{\wt}{\widetilde}
\newcommand{\eps}{\varepsilon}

\newcommand{\dd}{\mathrm d}
\newcommand{\ii}{\mathrm i}
\newcommand{\ee}{\mathrm e}

\begin{document}

\title[Ergodic Hankel operators]{Ergodic Hankel operators}

\author{Leonid Pastur}
\address{Department of Mathematics, King's College London, Strand, London, WC2R~2LS, United Kingdom}
\address{B. Verkin Institute for Low Temperature Physics and Engineering, 47 Nauky Avenue, Kharkiv, 61103 Ukraine}
\email{leonid.pastur@kcl.ac.uk}

\author{Alexander Pushnitski}
\address{Department of Mathematics, King's College London, Strand, London, WC2R~2LS, United Kingdom}
\email{alexander.pushnitski@kcl.ac.uk}



\date{25 September 2025}

\begin{abstract}
We introduce a new class of operators: ergodic families of self-adjoint Hankel operators realised as integral operators on the half-line. Inspired by the spectral theory of differential and finite-difference operators with ergodic coefficients, we develop a spectral theory of ergodic Hankel operators. For these operators, we define the Integrated Density of States (IDS) measure and establish its fundamental properties. In particular, we determine the total mass of the IDS measure in the positive semi-definite case. We also consider in more detail two classes of ergodic Hankel operators for which we are able to make further progress: periodic Hankel operators and the random Kronig--Penney--Hankel (rKPH) model. For periodic Hankel operators, we prove that the IDS measure is a sum of a pure point and absolutely continuous components, and describe the structure of both components. For the rKPH model, we prove the counterparts of the cornerstone results of the spectral theory of random Schr\"odinger operators:  Lifshitz tails at the edges of the spectrum, the Wegner bound and Anderson localisation in a natural asymptotic regime. We conclude with some open problems. 
\end{abstract}

\maketitle

\setcounter{tocdepth}{1}
\tableofcontents

\section{Introduction}
\label{sec.a}

\subsection{Hankel operators}
\label{sec:a1}
Hankel operators can be represented in two different (but unitarily equivalent) forms: as infinite matrices acting on $\ell^2(\bbZ_+)$,
\begin{equation}
\ell^2(\bbZ_+)\ni x\mapsto \sum_{m=0}^\infty h_{n+m}x_m, \quad n\in\bbZ_+
\label{eq:a0}
\end{equation}
for some sequence $\{h_n\}_{n=0}^\infty$ or as integral operators acting on $L^2(\bbR_+)$, 
\begin{equation}
H: L^2(\bbR_+)\ni f\mapsto \int_0^\infty h(t+s)f(s)\dd s,\quad t>0,
\label{eq:a12}
\end{equation}
for some \emph{kernel function} $h$ on $\bbR_+$. 
See e.g. \cite[Chapter 1]{Peller} or \cite[Part B, Chapter 1]{Nikolski} for the details. There are also descriptions of Hankel operators in terms of Hardy spaces, but these are outside the scope of our work.

In this paper, we will work with  \eqref{eq:a12}, because we are interested in the interaction of Hankel operators with the unitary group of dilations acting on $L^2(\bbR_+)$, 
\[
(D_\tau f)(t)=e^{\tau/2}f(e^\tau t), \quad t>0,
\]
where $\tau\in\bbR$. Throughout the paper, $H$ will denote the Hankel operator on $L^2(\bbR_+)$ defined by \eqref{eq:a12}. In general, the kernel function $h$ may be a distribution (see  \cite[Chapter 1]{Peller}), but in concrete situations that are of interest to us below, it will be a  measurable function on $\bbR_+$, sufficiently regular away from zero. 

We will only be interested in \emph{bounded} and \emph{self-adjoint} operators ${H}$. 
We will discuss boundedness in a few lines below. Since the integral kernel $h(t+s)$ of ${H}$ is symmetric, self-adjointness is equivalent to $h$ being real-valued.
 
As a warm-up, we discuss three examples. 

\begin{example}[Finite rank operators]\label{exa:r1}
If $h(t)=\ee^{-\alpha t}$ with some $\alpha>0$, then ${H}$ is a rank one operator. 
Of course, if $h$ is a finite linear combination of such exponentials with real coefficients, then $H$ is a  finite rank self-adjoint operator. 
\end{example}

\begin{example}[Carleman operator]\label{exa:Carl}
If $h(t)=1/t$, the operator ${H}$ is known as the \emph{Carleman operator}. This operator is self-adjoint and bounded but not compact. Its spectrum is purely absolutely continuous, coincides with the interval $[0,\pi]$ and has multiplicity two (see e.g. \cite[Section~10.2]{Peller}). 
\end{example}

We pause to display the boundedness of the Carleman operator as the estimate
\begin{equation}
\Abs{\int_0^\infty \int_0^\infty\frac{f(t)\overline{f(s)}}{t+s}dt\, ds}
\leq
\pi \norm{f}_{L^2(\bbR_+)}^2
\label{eq:a7}
\end{equation}
for future reference. 

Throughout the paper, we will assume that our kernel functions satisfy the bound 
\begin{equation}
\abs{h(t)}\leq C_h/t,\quad t>0,
\label{eq:a1}
\end{equation}
with some $t$-independent $C_h<\infty$. 
From here by \eqref{eq:a7} it immediately follows that ${H}$ is bounded on $L^2(\bbR_+)$, with the operator norm bound 
\[
\norm{{H}}\leq \pi C_h.
\]
Note that condition \eqref{eq:a1} is sufficient, but not necessary for the boundedness of ${H}$. A necessary and sufficient condition is given by Nehari's theorem, see e.g. \cite[Section~1.8]{Peller}.

\begin{example}[The Mathieu--Hankel operator and periodic operators]
Consider the kernel function 
\[
h(t)=\frac{a+\cos(\log t)}{t}, \quad t\in\bbR_+
\]
with $a\in\bbR$. The corresponding Hankel operator was introduced in \cite{PuSobolev} and dubbed the \emph{Mathieu--Hankel operator}, by analogy with the Mathieu operator
\[
-\frac{\dd^2}{\dd x^2}+\cos x, \quad x\in\bbR.
\]
The spectral analysis of the Mathieu--Hankel operator turns out to be an interesting and rich problem.  This operator commutes with dilations $D_{2\pi}$ by the factor $\ee^{2\pi}$ in $L^2(\bbR_+)$. Consequently, its spectrum has a band structure, with bands accumulating to zero, and some bands may degenerate into points (depending on the parameter $a$). 

More generally, let 
\begin{equation}
h(t)=\frac{P(\log t)}{t},
\label{eq:a14}
\end{equation}
where $P$ is a $\tau$-periodic real-valued function on $\bbR$, 
\begin{equation}
P(\xi)=P(\xi+\tau), \quad \xi\in\bbR.
\label{eq:a14a}
\end{equation}
Then the corresponding operator ${H}$ commutes with dilations by the factor $\ee^{\tau}$:
\begin{equation}
{H}D_\tau=D_\tau{H}.
\label{eq:a6}
\end{equation}
In \cite{PuSobolev}, such operators were termed \emph{periodic Hankel operators} and a detailed spectral theory was developed for them, which exhibits numerous parallels to the Floquet--Bloch theory of periodic Schr\"odinger operators. We will say more about periodic Hankel operators in Sections~\ref{sec.d} and \ref{sec.dd} below. 
\end{example}

\subsection{Ergodic Hankel operators}
Motivated in part by periodic Hankel operators, in this paper we consider \emph{ergodic Hankel operators ${H}_\omega$.} Here we briefly and informally explain our framework; precise definitions are given in Section~\ref{sec.b} below. 

Let $T$ be an ergodic measure-preserving automorphism on a probability space $\Omega$. We consider families $\{H_\omega\}_{\omega\in\Omega}$ of bounded self-adjoint Hankel operators satisfying the ergodicity condition 
\begin{equation}
H_{T\omega}D_\tau=D_\tau H_\omega
\label{eq:eHa}
\end{equation}
for some $\tau>0$ (compare with \eqref{eq:a6}). We will refer to the family $\{H_\omega\}_{\omega\in\Omega}$ as an \emph{ergodic Hankel operator} (we use a singular form, following the pattern of \emph{random variable}). Equivalently, ergodic Hankel operators are operators with kernel functions
\[
h_\omega(t)=\frac{P_\omega(\log t)}{t}, 
\]
where $P_\omega$ is a real-valued ergodic process on the real line, i.e. 
\[
P_{T\omega}(\xi)=P_\omega(\xi+\tau)
\]
(compare with \eqref{eq:a14a}). 
We note that assumption \eqref{eq:a1} is equivalent to 
\begin{equation}
\sup_{\xi\in\bbR}\abs{P_\omega(\xi)}\leq C_h. 
\label{eq:a1a}
\end{equation}

Standard theory of ergodic operators (see \cite[Section 1.D]{Pa-Fi:92}, where the term \emph{metrically transitive} instead of \emph{ergodic} is used) is usually presented for operators acting on $L^2(\bbR)$ and satisfying the relation of the type \eqref{eq:eHa} with $D_\tau$ replaced by the shift operator in $L^2(\bbR)$, 
\begin{equation}
(U_\tau f)(\xi)=f(\xi+\tau), \quad x\in\bbR. 
\label{eq:b3}
\end{equation}
Because of this, we find it convenient to ``transplant'' Hankel operators to $L^2(\bbR)$ by using the exponential change of variable $t=\ee^\xi$. Let $\calE: L^2(\bbR_+)\to L^2(\bbR)$ be the unitary operator
\[
(\calE f)(\xi)=\ee^{\xi/2}f(\ee^\xi), \quad \xi\in\bbR.
\]
For a bounded Hankel operator ${H}$, we will consider the operator 
\begin{equation}
\bH=\calE{H} \calE^*\quad \text{ in $L^2(\bbR)$.}
\label{eq:KE}
\end{equation}
Throughout the paper, we assume that operators $H$ in $L^2(\bbR_+)$  and $\bH$ in $L^2(\bbR)$ are linked in this way, and we also write $\bH_\omega=\calE{H}_\omega \calE^*$ for an ergodic Hankel operator $H_\omega$. We will refer to $\bH$ (or $\bH_\omega$) as Hankel operators in $L^2(\bbR)$.

A direct computation shows that $\bH$ has the integral kernel
\[
\bH(x,y)=\ee^{\frac{x+y}{2}}h(\ee^x+\ee^y),\quad x,y\in\bbR,
\]
where $h$ is the kernel function of ${H}$. In particular, if the kernel function is written as in \eqref{eq:a14} (with any $P$), then 
\begin{equation}
\bH(x,y)
=\frac{P(\log(\ee^x+\ee^y))}{2\cosh\frac{y-x}{2}}
=\frac{P(x+\log(1+\ee^{y-x}))}{2\cosh\frac{y-x}{2}}. 
\label{eq:a2a}
\end{equation}

Since $\calE$ maps dilations into shifts $U_\tau=\calE D_\tau\calE^*$,
ergodic Hankel operators $\bH_\omega$ in $L^2(\bbR)$ satisfy 
\[
\bH_{T\omega}U_\tau=U_\tau \bH_\omega,
\]
which allows the application of the standard version of ergodic spectral theory in $L^2(\bbR)$.

\subsection{The structure of the paper and overview of main results}
In Section~\ref{sec.aa} we explain that positive semi-definite Hankel operators play a special role in the theory, as the most tractable class with remarkable properties. Throughout the paper, we use the term \emph{positive} to mean \emph{positive semi-definite}. 

In Section~\ref{sec.bb} we recall necessary background from ergodic theory in the form that we will use. In Section~\ref{sec.b} we define ergodic Hankel operators and discuss some examples in general terms. We will see that as a consequence of general theory of ergodic operators, the spectrum of an ergodic Hankel operator $\bH_\omega$ is deterministic (i.e. independent of $\omega$ almost surely).

In Section~\ref{sec.c} we introduce the \emph{Integrated Density of States} (IDS) measure $\nu$ of $\bH_\omega$, an important spectral characteristic of ergodic operators. 
In Section~\ref{sec:szego} we give two alternative characterisations of this measure. 
In Section~\ref{sec.cc} we study the IDS measure for positive Hankel operators and in particular determine the total mass of this measure. We also prove that for positive ergodic Hankel operators, the IDS measure is continuous, while in general (in the absence of positivity) it may be a pure point measure.

The rest of the paper is devoted to two classes of Hankel operators where we are able to make further progress: periodic Hankel operators (Sections~\ref{sec.d} and \ref{sec.dd}) and random Hankel operators (Section~\ref{sec.g}). These classes represent two ``extreme points'' (in terms of the ``amount'' of randomness) of the broad class of ergodic operators. 

For periodic Hankel operators, we prove that the IDS measure is a sum of a pure point and an absolutely continuous components, and describe the structure of both components. For random Hankel operators, we focus on the class that we call \emph{random Kronig--Penney--Hankel model} (rKPH) by analogy with the Kronig--Penney model in the theory of Schr\"odinger operators, see e.g. \cite{Al-Co:12}. For the rKPH model, we establish the three basic facts of spectral theory of ergodic operators:  the Lifshitz tails asymptotics at spectral edges, the Wegner estimate and Anderson type localisation (a.s. point spectrum) in a natural asymptotic regime. Here we heavily rely on the results of our closely related paper \cite{Pa-Pu}.

\subsection{Hankel and Schr\"odinger}
There are numerous similarities between the spectral theory of Hankel operators and the spectral theory of Schr\"odinger operators. These similarities were first mentioned in print by J.~Howland \cite{Howland}. More recently they were advertised by D.~Yafaev, both in print \cite{Ya3} and in seminar talks. While the origin of these similarities remains somewhat mysterious (at least to the present authors), it is becoming clear that they involve, on one hand, \emph{positive} Hankel operators and on the other hand Schr\"odinger operators \emph{on the real line,}
\begin{equation}
-\frac{\dd^2}{\dd x^2}+q(x), \quad x\in\bbR.
\label{eq:Schrod}
\end{equation}
In particular, the results (but not the technique!) of the spectral theory of positive periodic Hankel operators, developed in \cite{PuSobolev}, are to a large extent identical to the results of the spectral theory of periodic  Schr\"odinger operators \eqref{eq:Schrod}. The results of this paper confirm and extend this analogy. 

There are, however, two notable aspects where ``Hankel theory'' differs from ``Schr\"odinger theory''. The first one is that the kernels of Hankel operators are either trivial or infinite dimensional, and the spectral analysis of these two cases proceeds in two different ways. There is no obvious analogy to this phenomenon in the ``Schr\"odinger theory''. 

The second aspect is that Schr\"odinger operators \eqref{eq:Schrod} are always unbounded, while Hankel operators can be both bounded and unbounded, and in fact interesting features of the theory appear already for bounded operators. (In this paper, we focus exclusively on bounded Hankel operators, although it is possible to work with unbounded operators as well, see e.g. \cite{PuTreil1,Ya4}.) Moreover, there is some apparent analogy between the spectral properties of Schr\"odinger operators at infinity and Hankel operators at zero. In this respect, it is more appropriate to compare Hankel operators with \emph{resolvents} of Schr\"odinger operators.

At the technical level, the most obvious fundamental difficulty of the ``Hankel theory'' is that there is currently no machinery to interpret the eigenvalue equation for Hankel operators in a language similar to the language of differential equations for Schr\"odinger operators. Thus, the notion of a Cauchy solution to the differential equation and associated efficient techniques (Sturm oscillation theorems, transfer matrices, Lyapunov exponents, $m$-functions, see e.g. \cite[Chapters III, V, VII]{Pa-Fi:92}) is not available for Hankel operators.

Finally, we note that non-positive Hankel operators are much more complicated than the positive ones. For example (see Section~\ref{sec:d6}), they can have non-zero eigenvalues of infinite multiplicity. In our view, comparing them to Schr\"odinger operators would not be accurate.

\subsection{Notation}
If $A$ is an integral operator in $L^2(\bbR)$ (or in $L^2(\bbR_+)$), we denote by $A(x,y)$, $x,y\in\bbR$ (or $x,y\in\bbR_+$), the integral kernel of $A$. We denote the inner product of elements $f$ and $g$ of a Hilbert space by $\jap{f,g}$, with a subscript indicating the Hilbert space if necessary. Our inner products are linear in $f$ and anti-linear in $g$. 
For a set $\Delta\subset\bbR$, we denote by $\chi_\Delta$ the characteristic function of $\Delta$. Throughout the text, $C$ denotes a generic constant in estimates, independent of the essential variables, whose value may change from one line to the next.

\subsection{Acknowledgements}
We are grateful to Alexander Sobolev for useful discussions.

\section{Positive Hankel operators}
\label{sec.aa}

\subsection{The Megretskii--Peller--Treil theorem}
In order to explain the role of positive Hankel operators in the theory, we start by recalling the Megretskii--Peller--Treil theorem. 

\begin{theorem}\cite[Theorem~1]{MPT}\label{thm.MPT}
A bounded self-adjoint operator ${H}$ with a scalar spectral measure $\mu$ and the spectral multiplicity function $m$ is unitarily equivalent to a Hankel operator if and only if the following conditions are satisfied:
\begin{enumerate}[\rm (i)]
\item
either $\Ker {H}=\{0\}$ or $\dim\Ker {H}=\infty$;
\item
${H}$ is not invertible;
\item
$\abs{m(t)-m(-t)}\leq2$ almost everywhere with respect to the absolutely continuous part of $\mu$;
\item
$\abs{m(t)-m(-t)}\leq1$ almost everywhere with respect to the singular part of $\mu$. 
\end{enumerate}
\end{theorem}
We note that the proof of the ``if'' part of this theorem is very difficult. 
Of importance to us is the following consequence of the ``only if'' part of the theorem.

\begin{corollary}\label{crl:MPT}
Let ${H}$ be a bounded positive Hankel operator. Then the absolutely continuous spectrum of the restriction ${H}|_{(\Ker {H})^\perp}$ has multiplicity $\leq2$, while the singular spectrum (in particular, the point spectrum) of ${H}|_{(\Ker {H})^\perp}$ has multiplicity one.  
\end{corollary}

This corollary indicates one more similarity between positive Hankel operators (restricted to the orthogonal complement to its kernel) and the full-line Schr\"odinger operators \eqref{eq:Schrod}. For the corresponding statement on the spectral multiplicity of \eqref{eq:Schrod}, see, e.g.  \cite[Theorem A]{Kac}. 

\subsection{Positive Hankel operators in $L^2(\bbR_+)$}\label{sec:a2}
A bounded Hankel operator ${H}$ in $L^2(\bbR_+)$ (see \eqref{eq:a12}) is positive  if and only if the corresponding kernel function $h$ can be represented as the Laplace transform of a (positive) measure $\sigma$ on $\bbR_+$
\begin{equation}
h(t)=\int_0^\infty \ee^{-tx}\dd\sigma(x), \quad t>0,
\label{eq:a8}
\end{equation}
where $\sigma(\{0\})=0$ and the Carleson condition 
\begin{equation}
\sigma((0,a))\leq C_\sigma a, \quad \forall a>0, 
\label{eq:a9}
\end{equation}
is satisfied with some $C_\sigma>0$. Of course, $\sigma$ is uniquely defined by $h$. A proof of the integral representation \eqref{eq:a8} and a detailed discussion can be found, for example, in \cite[Theorems 5.1 and 5.3]{Ya1}. The equivalence of the boundedness of ${H}$ and the Carleson condition was observed by Widom \cite{Widom66}. We note that a one-line argument with integration by parts shows that the Carleson condition \eqref{eq:a9} implies the bound \eqref{eq:a1} with $C_h=C_\sigma$. 
We also note that the representation \eqref{eq:a8} ensures that $h\in C^\infty(\bbR_+)$.

\begin{example}
Let $\sigma$ be a single point mass at $\alpha>0$ with the total mass $2\alpha$. Then $h(t)=2\alpha\,  e^{-\alpha t}$ and the corresponding Hankel operator ${H}$ has rank one and coincides with the orthogonal projection onto the one-dimensional subspace spanned by the function $e^{-\alpha t}$ in $L^2(\bbR_+)$ (cf. Example~\ref{exa:r1}). Of course, if $\sigma$ is a finite linear combination of point masses, then ${H}$ is a positive finite rank operator. 
\end{example}

\begin{example}
Let $\sigma$ be the Lebesgue measure restricted to $\bbR_+$. Then $h(t)=1/t$ and the corresponding Hankel operator is the Carleman operator.  
\end{example}

\subsection{Positive Hankel operators in $L^2(\bbR)$}\label{sec:aa3}
Let ${H}$ be a positive Hankel operator in $L^2(\bbR_+)$, and $\bH$ be the corresponding Hankel operator in $L^2(\bbR)$, see \eqref{eq:KE}. Let $\sigma$ be the measure in the integral representation \eqref{eq:a8}, and let $\Sigma$ be the measure on $\bbR$, related to $\sigma$ through the exponential change of variable $t=\ee^{-\xi}$ as follows:
\[
\Sigma((a,b))=\int_{\ee^{-b}}^{\ee^{-a}}\frac{\dd\sigma(t)}{t}. 
\]
For example, if $\sigma$ is absolutely continuous with the density $s$, i.e. $\dd\sigma(t)=s(t)\dd t$, then $\dd\Sigma(\xi)=S(\xi)\dd\xi$, where $S(\xi)=s(\ee^{-\xi})$.  Rewriting \eqref{eq:a8} in terms of $\Sigma$, we find 
\[
h(t)=\int_{-\infty}^\infty \ee^{-t\ee^{-\xi}}\ee^{-\xi}\dd\Sigma(\xi), \quad t>0.
\]
Furthermore, a simple computation shows that the integral kernel of $\bH$ is
\begin{equation}
\bH(x,y)=\int_{-\infty}^\infty \beta(x-\xi)\beta(y-\xi)\dd\Sigma(\xi), 
\quad\text{ where }\quad
\beta(\xi)=\ee^{-\ee^\xi}\ee^{\xi/2}. 
\label{eq:a4}
\end{equation}
We will often think of the measure $\Sigma$ on the real line as of the basic functional parameter defining the positive Hankel operator $\bH$, akin to the potential $q$ defining the Schr\"odinger operator \eqref{eq:Schrod}.

By an elementary calculation (see Appendix), the Carleson condition \eqref{eq:a9} is equivalent to the \emph{uniform local boundedness} condition 
\begin{equation}
\Sigma((x-1,x+1))\leq C_\Sigma, \quad x\in\bbR
\label{eq:a13}
\end{equation}
with some $C_\Sigma$ independent of $x$. 

If ${H}$ is a positive $\tau$-periodic Hankel operator \eqref{eq:a6}, then it is not difficult to check that the measure $\Sigma$ in the representation \eqref{eq:a4} is also $\tau$-periodic.

\subsection{A pseudodifferential point of view}
For the purposes of the informal discussion here, let us assume that the measure $\Sigma$ is absolutely continuous, $\dd\Sigma(\xi)=S(\xi)\dd\xi$. 
By \eqref{eq:a4}, the Hankel operator $\bH$ can be viewed as the product (right to left) of the convolution with $\beta(-x)$,  multiplication by $S(\xi)$ and the convolution with $\beta(x)$. 
Observe that 
\begin{align*}
\widehat \beta(u)&=\frac1{\sqrt{2\pi}}\int_{-\infty}^\infty \beta(x)\ee^{-\ii ux}\dd x
=\frac1{\sqrt{2\pi}}\int_{-\infty}^\infty \ee^{-\ee^x}\ee^{x/2}\ee^{-\ii ux}\dd x
\\
&=\frac1{\sqrt{2\pi}}\int_0^\infty \ee^{-\lambda}\lambda^{-\frac12-\ii u}\dd\lambda
=\frac1{\sqrt{2\pi}}\Gamma(\tfrac12-\ii u),
\end{align*}
with the Gamma function in the right-hand side. This shows that the operator of 
convolution with $\beta$ can be written as the pseudodifferential operator $\Gamma(\frac12-\frac{\dd}{\dd x})$. 
Putting this together, we find that $\bH$ can be viewed as the pseudodifferential operator 
\begin{equation}
\bH=\Gamma(\tfrac12-\tfrac{\dd}{\dd x})\Sigma\Gamma(\tfrac12-\tfrac{\dd}{\dd x})^*,
\label{eq:a5}
\end{equation}
in $L^2(\bbR)$, where $\Sigma$ denotes the operator of multiplication by the density $S$.

This representation shows that the theory of positive Hankel operators can be viewed as the theory of a very special class of pseudodifferential operators. Note that the operator $\Gamma(\tfrac12-\tfrac{\dd}{\dd x})$ is extremely ``smooth'', i.e. it is a convolution with the Schwartz class function $\beta$, while the operator $\Sigma$ is in general singular, i.e. it represents a measure. This is, of course, quite far from the classical theory of pseudodifferential operators. Regardless of these technical details, the pseudodifferential representation \eqref{eq:a5} for $\bH$ can be quite useful. It has been studied systematically by Yafaev \cite{yaf_apde_2015,Ya1} and an equivalent form of it goes back to Widom \cite{Widom66}. 

We will not use the pseudodifferential form \eqref{eq:a5} explicitly, but it turns out to be a source of insight and guidance in spectral theory of Hankel operators.

\section{Minimal background on ergodic theory}
\label{sec.bb}

In this section, we briefly recall necessary definitions and concepts of ergodic spectral theory adapted for our purposes; for details, see e.g. \cite{Aiz-War,Da-Fi:24,Pa-Fi:92}.

\subsection{Ergodic automorphisms}\label{sec:bb1}
Let $(\Omega,\calF,\bbP)$ be a probability space, where $\Omega$ is a set, $\calF$ is a sigma-algebra of measurable subsets on $\Omega$ and $\bbP$ is a probability measure defined on $\calF$. We denote by $\bE$ the expectation operation associated with $\bbP$, and write \emph{almost surely} (a.s.) for a proposition that holds true on a subset of $\Omega$ of full measure.

We will consider families of \emph{measure preserving} automorphisms of $\Omega$. These families will be parametrised either by $\bbR$ (the \emph{continuous case}) or by $\bbZ$ (the \emph{discrete case}). 
We denote these families of automorphisms by $\{T_a\}$, where $a\in\bbR$ or $a\in\bbZ$. In both discrete and continuous cases, we assume that $\{T_a\}$ is a \emph{group} and that $\{T_a\}$ are \emph{ergodic}, i.e. if a measurable set $\Delta\subset\Omega$ satisfies $T_a^{-1}(\Delta)=\Delta$ for all $a$, then $\Delta$ has measure $0$ or $1$. 

In the discrete case, the group structure implies that $T_a=(T_1)^a$, $a\in\bbZ$, and so it suffices to check the ergodicity condition for a single transformation $T=T_1$.

\begin{example}\label{exa:c1}
Let $\Omega=\bbT=\bbR/\bbZ$ be the unit circle with the normalised Lebesgue measure. Then the rotations 
\[
T_a \omega=\omega+a, \quad a\in\bbR
\]
is a group of ergodic automorphisms (continuous case). 
\end{example}

\begin{example}\label{exa:c2}
Let $\Omega=\bbT^d$ be the $d$-dimensional torus with the usual Lebesgue measure. Let $\alpha_1,\dots,\alpha_d$ be \emph{rationally independent} real numbers. Then the rotations
\[
T_a(\omega_1,\dots,\omega_d)=(\omega_1+\alpha_1 a, \dots,\omega_d+\alpha_d a), \quad a\in\bbR
\]
is a group of ergodic automorphisms (continuous case). See e.g. \cite[Section 3.1]{CSF}.
\end{example}

\begin{example}\label{exa:c3}
Let $\bbP_0$ be a probability measure on $\bbR$, and $\Omega$ be the infinite product $\Omega=\bbR^{\bbZ}$ with the product measure $\bbP=\bbP_0^{\bbZ}$. The shift transformation 
\[
(T\omega)_n=\omega_{n+1}
\]
on $\Omega$ is ergodic, and  $\{T^a\}_{a\in\bbZ}$ is a group of ergodic automorphisms (discrete case). 
\end{example}

Note that Examples~\ref{exa:c1} and \ref{exa:c3} are the ``extreme points'' with respect to the ``amount'' of randomness. 

\subsection{Ergodic processes}
A family $\{f_\omega\}_{\omega\in\Omega}$ of complex-valued functions  on $\bbR$ is called an  \emph{ergodic process}, if for any $\xi\in\bbR$ the function $\omega\mapsto f_\omega(\xi)$ is measurable and 
\begin{equation}
f_{T_a\omega}(\xi)=f_\omega(\xi+a), \quad \xi\in\bbR,
\label{eq:b2}
\end{equation}
for any $a\in\bbR$ in the continuous  case and for any $a\in\bbZ$ in the discrete  case.  It will be convenient to introduce an additional parameter $\tau>0$ (we will refer to it as a \emph{period}) and instead of \eqref{eq:b2} require 
\begin{equation}
f_{T_a\omega}(\xi)=f_\omega(\xi+\tau a), \quad \xi\in\bbR.
\label{eq:b2-tau}
\end{equation}
Of course, $\tau$ can be absorbed into $a$ in the continuous case but not in the discrete case. In order to keep the exposition in the discrete and continuous cases parallel to each other, we will keep the period $\tau$ in both cases.

In the discrete case, it suffices to require \eqref{eq:b2-tau} for $a=1$, i.e. 
\begin{equation}
f_{T\omega}(\xi)=f_\omega(\xi+\tau), \quad \xi\in\bbR.
\label{eq:b2-1}
\end{equation}
In the continuous case, one can give an alternative version of definition \eqref{eq:b2-tau}. Let $F:\Omega\to\bbR$ be a measurable function (called a \emph{sample function} in this context). Set 
\[
f_\omega(\xi)=F(T_{\xi/\tau}\omega).
\]
Then $f_\omega$ satisfies \eqref{eq:b2-tau} and $f_\omega(0)=F(\omega)$.

In a similar way, a family of Borel measures $\Sigma_\omega$ on $\bbR$ is called an \emph{ergodic family} if for any Borel set $\Delta\subset \bbR$ the function $\omega\mapsto \Sigma_\omega(\Delta)$ is measurable, and if 
\[
\Sigma_{T_a\omega}(\Delta)=\Sigma_\omega(\Delta+\tau a)
\]
for any $a$. Here $\Delta+\tau a$ is the shifted set
\[
\Delta+\tau a=\{\xi+\tau a \in\bbR: \xi\in\Delta\}.
\]

Finally, we define an ergodic sequence; this definition makes sense in the discrete case only, and here $\tau=1$. A family of real-valued sequences $\varkappa_\omega=\{\varkappa_\omega(n)\}_{n\in\bbZ}$, parameterised by $\omega\in\Omega$, is called an \emph{ergodic sequence}, if for any $n\in\bbZ$ the function $\omega\mapsto \varkappa_\omega(n)$ is measurable, and if 
\[
\varkappa_{T\omega}(n)=\varkappa_\omega(n+1), \quad n\in\bbZ,
\]
cf. \eqref{eq:b2-1} with $\tau=1$. 

\begin{example}\label{exa:c4}
Let $F$ be a measurable $1$-periodic function on $\bbR$.
In the context of Example~\ref{exa:c1} above, let 
\[
f_\omega(\xi)=F(\xi+\tau\omega), \quad \xi\in\bbR, \quad \omega\in\bbT.
\]
Then $\{f_\omega\}_{\omega\in\Omega}$ is an ergodic process (continuous case, period $\tau$). 
\end{example}

\begin{example}[Quasi-periodic functions]\label{exa:c5}
Let $F$ be a measurable function on the torus $\bbT^d$. In the context of Example~\ref{exa:c2} above, let 
\[
f_\omega(\xi)=F(\omega_1+\alpha_1 \xi, \dots,\omega_d+\alpha_d \xi),
\quad \xi\in\bbR, \quad \omega\in\bbT^d.
\]
Then $\{f_\omega\}_{\omega\in\bbT^d}$ is an ergodic process (continuous case, period $\tau=1$). 
\end{example}

\begin{example}\label{exa:c6}
In the context of Example~\ref{exa:c3} above, let $\{\varkappa_\omega(n)\}_{n\in\bbZ}$ be a collection of independent identically distributed random variables, with the distribution $\bbP_0$. Then $\{\varkappa_\omega(n)\}_{n\in\bbZ}$ is an ergodic sequence (discrete case, period $\tau=1$). 
\end{example}

\begin{example}
Let $v$ be a continuous compactly supported function on $\bbR$ and let $\{\varkappa_\omega(n)\}_{n\in\bbZ}$ be as in the previous example. For $\tau>0$, define 
\[
f_\omega(\xi)=\sum_{n=-\infty}^\infty \varkappa_\omega(n)v(\xi-\tau n), \quad \xi\in\bbR, \quad \omega\in\Omega=\bbR^{\bbZ}.
\]
Then $\{f_\omega\}_{\omega\in\bbR^{\bbZ}}$ is an ergodic process (discrete case with period $\tau$). 
\end{example}

\begin{remark}\label{rmk.bdd}
It is clear that for an ergodic process $f_\omega$, the $L^\infty$-norm 
\[
\text{ess\,sup}_{\xi\in\bbR}\abs{f_\omega(\xi)}
\] 
is invariant under the transformations $T_a$, and therefore is independent of $\omega$ almost surely. Thus, for an ergodic process a.s. boundedness is equivalent to uniform boundedness. The same remark pertains to ergodic sequences. 
\end{remark}

\subsection{Ergodic theorem}
We will need two versions of the ergodic theorem. The first one is the fundamental Birkhoff--Khinchin Ergodic theorem, see e.g. \cite[Theorem~1.2.1]{CSF}.

\begin{theorem}[Ergodic Theorem, version 1]
\label{thm:erg1}
Let $F$ be a measurable function on $\Omega$ with $\bE\{\abs{F(\omega)}\}$ finite. Then almost surely, 
\[
\lim_{M\to\infty}\frac1{2M}\int_{-M}^M F(T_a\omega)\dd a=\bE\{F(\omega)\}
\]
in the continuous case and 
\[
\lim_{N\to\infty}\frac{1}{2N+1}\sum_{n=-N}^N F(T^n\omega)=\bE\{F(\omega)\}
\]
in the discrete case. 
\end{theorem}
The second version 
is the direct consequence of the previous one applied to ergodic processes. 

\begin{theorem}[Ergodic Theorem, version 2]
\label{thm:erg2}
\begin{enumerate}[\rm (i)]
\item
Continuous case: 
Let $f_\omega$ be an ergodic process such that $\bE\{\abs{f_\omega(0)}\}<\infty$.
Then almost surely
\[
\lim_{M\to\infty}\frac1{2M}\int_{-M}^M f_\omega(x)\dd x=\bE\{f_\omega(0)\}.
\]
\item
Discrete case with period $\tau$: 
Let $f_\omega$ be an ergodic process such that
\[
\int_0^\tau \bE\{\abs{f_\omega(x)}\}\dd x<\infty.
\]
Then almost surely
\[
\lim_{M\to\infty}\frac1{2M}\int_{-M}^M f_\omega(x)\dd x=\frac1\tau\int_0^\tau \bE\{f_\omega(x)\}\dd x.
\]
\end{enumerate}
\end{theorem}
See \cite[Proposition~1.13]{Pa-Fi:92} for the continuous case and the discussion in \cite{Pa-Fi:92} at the end of Section 1.15(d), pages 28--29 for the discrete case.

\section{Ergodic Hankel operators}
\label{sec.b}

\subsection{Ergodic Hankel operators: definition}
Below $U_a$ is the shift operator \eqref{eq:b3} in $L^2(\bbR)$. 
Throughout the paper we consider only bounded operators. In what follows, we assume that we are in the framework of Section~\ref{sec:bb1}, i.e. we have a probability space $(\Omega,\calF,\bbP)$ and an ergodic group of transformations $\{T_a\}$ on $\Omega$, with $a\in\bbR$ (continuous case) or $a\in\bbZ$ (discrete case). 

\begin{definition}[Ergodic operators]\label{def:ergodic}
Let  $\bH_\omega$, $\omega\in\Omega$, be a measurable family of bounded operators in $L^2(\bbR)$. (Measurability means that for any $f,g\in L^2(\bbR_+)$, the inner product $\jap{\bH_\omega f,g}$ is a measurable function of $\omega\in\Omega$.) 
One says that $\bH_\omega$ is an \emph{ergodic operator} (with a period $\tau>0$) if 
\begin{equation}
\bH_{T_a\omega}U_{\tau a}=U_{\tau a}\bH_\omega 
\label{eq:b5}
\end{equation}
for all $a\in\bbR$ in the continuous case and for all $a\in\bbZ$ in the discrete case. See \cite[Section 1.D]{Pa-Fi:92}. 
\end{definition}
In the discrete case, when $T_a=T^a$, it suffices to require 
\[
\bH_{T\omega}U_{\tau}=U_{\tau}\bH_\omega. 
\]
\begin{remark}
Similarly to Remark~\ref{rmk.bdd} above, it is easy to see that if $\bH_\omega$ is an ergodic operator which is bounded almost surely, then the norm of $\bH_\omega$ is independent of $\omega$ (see Proposition~\ref{prp:b7} below) and therefore we have almost surely $\norm{\bH_\omega}= C$ with $C$ independent of $\omega$. Thus, for ergodic operators a.s. boundedness is equivalent to uniform boundedness. 
\end{remark}

We will be interested in \emph{bounded self-adjoint ergodic} Hankel operators $\bH_\omega$. For Hankel operators, it is easy to rephrase condition \eqref{eq:b5} of the above definition in terms of the corresponding kernel functions $h_\omega$. For clarity, we state this as a simple proposition. For simplicity, we only discuss the kernel functions satisfying the bound  \eqref{eq:a1}. 

\begin{proposition}
Let $\bH_\omega$ be an ergodic Hankel operator with the kernel function $h_\omega$ satisfying the inequality \eqref{eq:a1} almost surely with $C_h$ independent of $\omega$. Then 
\begin{equation}
h_\omega(t)=\frac{P_\omega(\log t)}{t}, \quad t>0,
\label{eq:x1}
\end{equation}
where $P_\omega$ is an ergodic process (see \eqref{eq:b2-tau}) satisfying the inequality \eqref{eq:a1a} almost surely. 
Conversely, for any ergodic process $P_\omega$ satisfying \eqref{eq:a1a} almost surely,
the kernel function \eqref{eq:x1} generates an ergodic Hankel operator with the kernel function $h_\omega$ satisfying \eqref{eq:a1} almost surely. 
\end{proposition}
In brief, we can summarise this proposition as 
\[
\bH_{T_a\omega}U_{\tau a}=U_{\tau a}\bH_\omega
\quad\text{is equivalent to}\quad
P_{T_a\omega}(\xi)=P_\omega(\xi+\tau a).
\]

As a consequence of the discussion in Section~\ref{sec:aa3}, for \emph{positive} Hankel operators we can state the following version of the above proposition. 

\begin{proposition}\label{thm:b2}
Let $\bH_\omega$ be an ergodic Hankel operator which is positive and bounded almost surely. Then the corresponding kernel function $h_\omega$ can be represented as 
\[
h_\omega(t)=\int_{-\infty}^\infty \ee^{-t\ee^{-\xi}}\ee^{-\xi}\dd\Sigma_\omega(\xi)
\]
with a unique ergodic family of measures $\Sigma_\omega$. This family satisfies the uniform local boundedness condition \eqref{eq:a13} where $C_\Sigma$ is independent of $\omega$. 
The kernel function $h_\omega$ is $C^\infty$-smooth on $\bbR_+$ and satisfies the bound \eqref{eq:a1} almost surely with $C_h$ independent of $\omega$.
We have the representation \eqref{eq:a4} for the integral kernel of $\bH_\omega$:
\begin{equation}
\bH_\omega(x,y)=\int_{-\infty}^\infty \beta(x-\xi)\beta(y-\xi)\dd\Sigma_\omega(\xi), \quad x,y\in\bbR,
\label{eq:a4omega}
\end{equation}
where $\beta(\xi)=\ee^{-\ee^\xi}\ee^{\xi/2}$. 
Conversely, any ergodic family of measures $\Sigma_\omega$ satisfying \eqref{eq:a13} generates a positive bounded ergodic Hankel operator in this way. 
\end{proposition}
In brief, 
\begin{equation}
\bH_{T_a\omega}U_{\tau a}=U_{\tau a}\bH_\omega
\quad\text{is equivalent to}\quad
\Sigma_{T_a\omega}(\Delta)=\Sigma_\omega(\Delta+\tau a).
\label{eq:add1}
\end{equation}

\subsection{Examples}

\begin{example}[Periodic operators]
\label{ex:b4}
Let us discuss $\tau$-periodic Hankel operators, i.e. the Hankel operators ${H}$ satisfying the commutation relation \eqref{eq:a6}. This commutation relation is equivalent to 
\begin{equation}
\bH U_\tau =U_\tau \bH
\label{eq:periodic}
\end{equation}
for  $\bH=\calE{H} \calE^*$. As it is standard (see e.g. \cite[Section 1.15(g)]{Pa-Fi:92}), periodic Hankel operators fit into the ergodic framework as follows. Let $\Omega=\bbT$ be the unit circle and let $\{T_a\}_{a\in\bbR}$ be the group of rotations of $\bbT$ (see Example~\ref{exa:c1}). Given a $\tau$-periodic operator $\bH$ as in \eqref{eq:periodic}, we define
\begin{equation}
\bH_\omega=U_{\tau \omega} \bH U_{\tau \omega}^*, \quad \omega\in\bbT.
\label{eq:x3}
\end{equation}
Then $\bH_\omega$ is an ergodic Hankel operator in the sense of Definition~\ref{def:ergodic}. Note that in this simplest case all operators $\bH_\omega$ are unitarily equivalent to $\bH=\bH_0$. 

This definition can be equivalently rewritten in terms of the kernel functions as follows. Let $P$ be a bounded, measureable $\tau$-periodic function on $\bbR$. Define the family of kernel functions 
\begin{equation}
h_\omega(t)=\frac{P(\tau\omega+\log t)}{t}, \quad t>0, \quad \omega\in\bbT.
\label{eq:per2}
\end{equation}
Then the corresponding Hankel operator $\bH_\omega$ is an ergodic operator. 

We will discuss periodic Hankel operators in detail in Section~\ref{sec.d} below. 
\end{example}

\begin{example}[Quasiperiodic operators]
Let $\{P_\omega\}_{\omega\in\bbT^d}$ be an ergodic process corresponding to (bounded, real-valued) quasiperiodic functions as in Example~\ref{exa:c5}. Setting $h_\omega(t)=P_\omega(\log t)/t$ gives rise to a class of ergodic Hankel operators $H_\omega$ that is natural to call quasiperiodic. We are not going to consider this class of operators in this paper. We only mention that there is a large and active branch of spectral theory studying quasiperiodic and, more generally, almost-periodic Schr\"odinger operators, see e.g. \cite{Pa-Fi:92,Ji:22,Da-Fi:24} and references therein. 
\end{example}

\begin{example}[rKPH model]
\label{ex:b3}
Let $\bbP_0$ be a probability measure on $\bbR$ with a compact support on the positive half-line, and let $\{\varkappa_\omega(n)\}_{n\in\bbZ}$ be a sequence of i.i.d. random variables as in Example~\ref{exa:c6}. 
Define the measure $\Sigma_\omega$ as
\[
\Sigma_\omega=\sum_{n\in\bbZ}\varkappa_\omega(n) \delta_{\tau n}, 
\]
where $\delta_x$ is a point mass at $x$. Since the sequence $\{\varkappa_\omega(n)\}_{n\in\bbZ}$ is bounded, it is clear that $\Sigma_\omega$ satisfies the uniform local  boundedness condition \eqref{eq:a13}. Moreover, $\{\Sigma_\omega\}_{\omega\in\bbR^{\bbZ}}$ is an ergodic family of measures (discrete case with period $\tau$):
\[
\Sigma_{T\omega}(\Delta)=\Sigma_{\omega}(\Delta+\tau),
\]
where $T$ is the shift as in Example~\ref{exa:c3}. 
The corresponding ergodic Hankel operator $\bH_\omega$ in $L^2(\bbR)$ is the infinite sum of rank one operators: 
\[
\bH_\omega=\sum_{n\in\bbZ}\varkappa_\omega(n) \jap{\cdot,\psi_n}\psi_n, \quad 
\psi_n(\xi)=\beta(\xi-\tau n),
\]
where $\beta$ is as in \eqref{eq:a4} and $\norm{\psi_n}^2=1/2$. It can be viewed as an analogue of the alloy type model for the random Schr\"odinger operator, see formula (1.31) and the subsequent discussion in \cite{Pa-Fi:92}. 

In order to give this family of operators a name, and also to emphasize the analogies with the random Kronig--Penney model in the theory of Schr\"odinger operators (see e.g. \cite{Al-Co:12}), we will call the above family $\bH_\omega$ the \emph{random Kronig--Penney--Hankel model}, or \emph{rKPH} for short. We will discuss this model in Section~\ref{sec.g} below. 

It is sometimes convenient to allow $\varkappa_\omega(n)$ to be negative; then $\Sigma_\omega$ will be a signed measure (see Section~\ref{sec:d6} below).
\end{example}

\subsection{The spectrum and the kernels}

As a consequence of general theory of ergodic operators, we have the following properties.
\begin{proposition}\cite[Section 2]{Pa-Fi:92}\label{prp:b7}
Let $\bH_\omega$ be a bounded ergodic self-adjoint Hankel operator. Then the spectrum of $\bH_\omega$ and its continuous, absolutely continuous, singular continuous and pure point components are all non-random sets. The spectral multiplicity function of $\bH_\omega$ is also non-random. 
\end{proposition}

The following statement is specific to Hankel operators and has no analogue in the ``Schr\"odinger theory''.

\begin{proposition}\label{prop:b8}
Let $\bH_\omega$ be a bounded ergodic self-adjoint Hankel operator. We have:
\begin{enumerate}[\rm (i)]
\item
either $\Ker\bH_\omega=\{0\}$ almost surely or $\dim\Ker\bH_\omega=\infty$ almost surely; 
\item
if $\bH_\omega$ is positive almost surely, then $\dim\Ker\bH_\omega=\infty$ if and only if  almost surely the measure $\Sigma_\omega$ (see \eqref{eq:a4})  is a pure point measure with the support satisfying 
\begin{equation}
\sum_{\xi\in\supp\Sigma_\omega}\sech\xi<\infty.
\label{eq:b9a}
\end{equation}
\end{enumerate}
In particular, the support of $\Sigma_\omega$ has no finite points of accumulation. 
\end{proposition}
\begin{remark}
For an ergodic family of measures $\Sigma_\omega$, consider the set of $\omega$ such that $\Sigma_\omega$ is pure point. By the ergodicity condition \eqref{eq:add1}, this set is invariant under all $T_a$. Thus, it has measure $0$ or $1$. It follows that either $\Sigma_\omega$ is pure point a.s. or it is not pure point a.s. Similarly, one checks that either \eqref{eq:b9a} is satisfied a.s. or not satisfied a.s. We note also that \eqref{eq:b9a}
is a condition on the support of $\Sigma_\omega$ only, i.e. it does not involve the weights (masses) of atoms of $\Sigma_\omega$. 
\end{remark}

\begin{proof}[Proof of Proposition~\ref{prop:b8}]
Part (i):
By Theorem~\ref{thm.MPT}(i), the function 
\[
\omega\mapsto\dim\Ker\bH_\omega
\]
takes values $0$ or $\infty$.  This function is invariant under the family $\{T_a\}$ of ergodic transformations of $\Omega$; hence it is a.s. constant.

Part (ii): 
it was proved in \cite[Theorem~2.5]{PuTreil1} that if ${H}$ is a positive Hankel operator with the measure $\sigma$ (see \eqref{eq:a8}), then $\dim\Ker{H}=\infty$ if and only if $\sigma$ is a pure point measure, supported on a sequence of points $\{a_n\}$ satisfying 
\begin{equation}
\sum_{n}\frac{a_n}{1+a_n^2}<\infty.
\label{eq:b9}
\end{equation}
Rewriting \eqref{eq:b9} in terms of $\Sigma_\omega$, we arrive at condition \eqref{eq:b9a}. 
\end{proof}
\begin{remark}
Condition \eqref{eq:b9} is the Blaschke condition for the zeros of a function in the Hardy class of the right half-plane. Thus, condition \eqref{eq:b9a} is the Blaschke condition in disguise. 
\end{remark}

\section{The IDS measure: definition}
\label{sec.c}

\subsection{Preliminaries}
Let $\bH_\omega$ be a bounded self-adjoint ergodic Hankel operator in $L^2(\bbR)$. We would like to define the integrated density of states (IDS) measure $\nu$ of $\bH_\omega$ as follows. For a Borel set $\Delta\subset\bbR$ separated away from the origin, let $\chi_\Delta(\bH_\omega)$ be the spectral projection of $\bH_\omega$ corresponding to $\Delta$, with the integral kernel $\chi_\Delta(\bH_\omega)(x,y)$, $x,y\in\bbR$. We would like to define
\begin{equation}
\nu(\Delta)=\bE\bigl\{\chi_\Delta(\bH_\omega)(0,0)\bigr\}
\label{eq:c38}
\end{equation}
in the continuous case, and 
\begin{equation}
\nu(\Delta)=\bE\left\{\frac1\tau\int_0^\tau\chi_\Delta(\bH_\omega)(x,x)\dd x\right\}
\label{eq:c39}
\end{equation}
in the discrete case with period $\tau$. In order for this definition to make sense, we need to ensure that the restriction of the integral kernel $\chi_\Delta(\bH_\omega)(x,y)$ onto a point $(0,0)$ or onto the diagonal $(x,x)$ is well-defined. Thus, we start with a simple technical lemma, and after that we will return to definitions \eqref{eq:c38} and \eqref{eq:c39}.

\subsection{Integral kernels}
\begin{lemma}\label{lma:c1}
Let $\bH$ be a self-adjoint Hankel operator in $L^2(\bbR)$ with the kernel function $h$ satisfying the bound \eqref{eq:a1} with some $C_h>0$. Let $\varphi$ be a Borel function on $\bbR$ satisfying the estimate 
\[
\abs{\varphi(\lambda)}\leq A\abs{\lambda}^2 \quad\text{ for }\quad \abs{\lambda}\leq \pi C_h
\] 
with some $A>0$. Then the operator $\varphi(\bH)$ in $L^2(\bbR)$ (see \eqref{eq:KE}) has the integral kernel $\Phi(x,y)$ which is jointly continuous in $x,y$ and bounded, with
\[
\abs{\Phi(x,y)}\leq AC_h^2.
\]
\end{lemma}
\begin{proof}
We write $\varphi(\lambda)=\widetilde\varphi(\lambda)\lambda^2$, where $\abs{\widetilde\varphi(\lambda)}\leq A$ for $\abs{\lambda}\leq \pi C_h$. Since $\norm{\bH}\leq \pi C_h$ (see \eqref{eq:a1}), this implies that 
\begin{equation}
\varphi(\bH)=X\bH^2, \quad \text{ with } \norm{X}\leq A. 
\label{eq:c22}
\end{equation}
Let us write $h(t)=P(\log t)/t$, with $\abs{P(\xi)}\leq C_h$. 
We recall that the integral kernel of $\bH$ is given by \eqref{eq:a2a}. Let us denote this kernel by 
\[
F_x(y):=\frac{P(\log(e^x+e^y))}{2\cosh\frac{x-y}{2}}, \quad x,y\in\bbR.
\]
Clearly, $F_x(y)=F_y(x)$. Since $P$ is bounded with $\abs{P}\leq C_h$, we see that $F_x\in L^2(\bbR)$ with a norm bound uniform in $x$: 
\begin{align*}
\int_{-\infty}^\infty \abs{F_x(y)}^2\dd y
&\leq
C_h^2
\int_{-\infty}^\infty \frac1{(2\cosh\frac{x-y}{2})^2}\dd y
=C_h^2.
\end{align*}
From here and \eqref{eq:c22} we see that the integral kernel $\Phi$ of $\varphi(\bH)$ can be written as
\[
\Phi(x,y)=\jap{XF_y,F_x}_{L^2(\bbR)}, \quad x,y\in\bbR. 
\]
In particular, this kernel is bounded:
\[
\abs{\Phi(x,y)}\leq \norm{X}\norm{F_x}\norm{F_y}\leq AC_h^2.
\]
It remains to prove that $\Phi(x,y)$ is continuous in $x$ and $y$. Since $X$ is bounded, it suffices to prove that $F_x$ is continuous in $x$ as an element of $L^2(\bbR)$. 

In order to simplify our notation, let us explain how to show that $\norm{F_x-F_0}_{L^2}\to0$ as $x\to0$. Let $x_n\to0$ as $n\to\infty$. We need to show that 
\[
\int_{-\infty}^\infty 
\Abs{\frac{P(\log(e^{x_n}+e^y))}{2\cosh\frac{y-x_n}{2}}-\frac{P(\log(1+e^y))}{2\cosh\frac{y}{2}}}^2\dd y\to0
\]
as $n\to\infty$. If $P$ is continuous, the statement immediately follows by dominated convergence. If $P$ is not continuous, this can be shown by a simple approximation argument using the boundedness of $P$ and the exponential decay of the factor $1/\cosh(y/2)$. We omit the details of this argument. 
\end{proof} 

\subsection{Definition of the IDS measure}
Now we are ready to define the IDS measure. 
\begin{definition}
Let $\bH_\omega$ be a bounded self-adjoint ergodic Hankel operator on $L^2(\bbR)$ (discrete or continuous case), satisfying \eqref{eq:a1} with $C_h$ independent of $\omega$. For a Borel set $\Delta\subset\bbR$ separated away from the origin, we define $\nu(\Delta)$ by \eqref{eq:c38} in the continuous case and \eqref{eq:c39} in the discrete case with period $\tau$. The measure $\nu$ is called the \emph{Integrated Density of States (IDS) measure} for $\bH_\omega$. 
\end{definition}

\begin{remark}
\begin{enumerate}[1.]
\item
Since the set $\Delta$ is assumed to be separated from the origin, the value $\nu(\{0\})$ remains undetermined by this definition. For definiteness, we set $\nu(\{0\})=0$.  
\item
The above definition ensures that the measure $\nu$ is sigma-finite on $\bbR\setminus\{0\}$, i.e. $\nu(\Delta)$ is finite for any bounded interval $\Delta$ separated from the origin. However, $\nu(\bbR)$ is not necessarily finite. We will come back to this discussion in Section~\ref{sec.cc}. 
\end{enumerate}
\end{remark}

As it is standard, we have 

\begin{proposition}\cite[Theorem~3.1]{Pa-Fi:92}\label{thm:c4}
The support of $\nu$ coincides with the deterministic spectrum of $\bH_\omega$. 
\end{proposition}
\begin{proof}
Formally Theorem~3.1 of \cite{Pa-Fi:92} refers to ergodic Schr\"odinger operators, although the same arguments, word for word, apply to $\bH_\omega$. 
\end{proof}
Thus, the IDS measure determines the location and the geometry of the spectrum of $\bH_\omega$.

\subsection{The IDS measure as a weak limit of measures}
As above, let $\Delta$ be a Borel set separated away from the origin. Since $\chi_\Delta(\bH_\omega)$ is an ergodic operator (see the general Theorem I.2.7 in \cite{Pa-Fi:92}), its integral kernel on the diagonal $\chi_\Delta(\bH_\omega)(x,x)$ is an ergodic process. Thus, by the Ergodic Theorem (Theorem~\ref{thm:erg2}), we have, almost surely, 
\begin{equation}
\nu(\Delta)=\lim_{M\to\infty}\frac1{2M}\int_{-M}^M \chi_\Delta(\bH_\omega)(x,x)\dd x.
\label{eq:c25}
\end{equation}
In general, the set of $\omega$ where \eqref{eq:c25} holds depends on $\Delta$. However, if $\Delta$ is an interval with rational endpoints, by a countability argument the set of $\omega$ of full measure can be chosen uniformly over $\Delta$. 

We rephrase \eqref{eq:c25} as follows:
\begin{equation}
\nu(\Delta)=\lim_{M\to\infty}\frac1{2M}\Tr(\chi_{M}\chi_\Delta(\bH_\omega)\chi_{M}),
\label{eq:c26}
\end{equation}
where 
\begin{equation}
\chi_M:=\chi_{(-M,M)}
\label{eq:cm}
\end{equation}
is the operator of multiplication by the indicator function of the interval $(-M,M)$ in $L^2(\bbR)$. The trace exists by an argument similar to that of Lemma~\ref{lma:c1}. 

Integrating \eqref{eq:c26} with a continuous weight, we can also rephrase it in terms of the weak convergence of measures. We state this as a proposition for future reference. 

\begin{proposition}\label{thm:c5}
Let $\bH_\omega$ be a self-adjoint ergodic Hankel operator satisfying \eqref{eq:a1} with an $\omega$-independent $C_h$, and let $\nu$ be the IDS measure of $\bH_\omega$. Then almost surely, 
\begin{equation}
\lim_{M\to\infty}\frac1{2M}\Tr(\chi_{M}\varphi(\bH_\omega)\chi_{M})
=
\int_{-\infty}^\infty \varphi(\lambda)\dd\nu(\lambda)
\label{eq:c28}
\end{equation}
for any continuous function $\varphi$ compactly supported on $\bbR\setminus\{0\}$.
\end{proposition}

\subsection{Example: the Carleman operator}
As an important basic example, we compute the IDS measure of the Carleman operator (see Example~\ref{exa:Carl}), which we denote by ${H}_C$ here. Of course, there is no dependence on the parameter $\omega$ here; formally, we can regard ${H}_C$ as an  ergodic operator with $\Omega$ a singleton.
By \eqref{eq:a2a}, the operator $\bH_C=\calE{H}_C\calE^*$ is the convolution with the function $\tfrac12\sech\tfrac{x}{2}$. The IDS measure of convolution operators is easy to compute: it is just the pushforward of the Lebesgue measure on $\bbR$ by the symbol, see e.g. \cite[formula (3.13)]{Pa-Fi:92}. By the elementary formula 
\[
\frac12
\int_{-\infty}^\infty \ee^{-\ii x\xi}\sech\tfrac{x}{2}\, \dd x=\pi\sech\pi\xi, \quad \xi\in\bbR,
\]
we see that the symbol of $\bH_C$ is the function $\pi\sech\pi\xi$. It follows that the corresponding IDS measure $\nu_C$ can be computed as
\begin{align}
\int_\lambda^\infty \dd\nu_C(x)
&=\int_{-\infty}^\infty \chi_{(\lambda,\infty)}(\pi\sech\pi\xi)\dd\xi
=\frac{2}{\pi}\arcsech(\lambda/\pi)
\notag
\\
&=\frac{2}{\pi}\log\biggl(\frac{\pi}{\lambda}+\sqrt{\frac{\pi^2}{\lambda^2}-1}\biggr)
\label{eq:cc4}
\end{align}
for $0<\lambda<\pi$.
Thus, the IDS measure $\nu_C$ is supported on the interval $[0,\pi]$ and is absolutely continuous with the density 
\[
\dd\nu_C(\lambda)=\frac{\chi_{(0,\pi)}(\lambda)\dd\lambda}{\lambda\sqrt{\pi^2-\lambda^2}}.
\]
We emphasize the non-integrable singularity of this density as $\lambda\to0_+$. In other words, the measure $\nu_C$ is infinite. Moreover, it follows from \eqref{eq:cc4} that \begin{equation}
\nu_C((\lambda,\infty))=\frac{2}{\pi}\log\frac1\lambda + O(1), \quad \lambda\to0_+.
\label{eq:cc5}
\end{equation}

\begin{remark}
Widom in \cite{Widom66} computed the analogue of the integrated density of states for a class of Hankel matrices \eqref{eq:a0}, including Hilbert's matrix. See also subsequent work \cite{Fedele} on a wider class of Hankel matrices. 
\end{remark}

\subsection{Finiteness of moments of $\nu$}\label{sec:c6}
\begin{theorem}\label{thm:moments}
Let $\bH_\omega$ be a self-adjoint ergodic Hankel operator satisfying \eqref{eq:a1} with an $\omega$-independent $C_h$ and let $\nu$ be the IDS measure of $\bH_\omega$. Then the second moment of $\nu$ is finite. If $\bH_\omega$ is positive, then the first moment of $\nu$ is finite. 
\end{theorem}
\begin{proof}
If $\varphi$ is any function vanishing in the neighbourhood of the origin and satisfying $0\leq\varphi(\lambda)\leq C\lambda^2$, we have, almost surely, 
\begin{align*}
\int_{-\infty}^\infty \varphi(\lambda)\dd\nu(\lambda)
=
\lim_{M\to\infty}\frac1{2M}\Tr(\chi_M\varphi(\bH_\omega)\chi_M)
\leq
C\limsup_{M\to\infty}\frac1{2M}\Tr(\chi_M\bH_\omega^2\chi_M), 
\end{align*}
hence it suffices to check that $\limsup$ in the right-hand side is finite. Using \eqref{eq:a2a}, we find 
\begin{align*}
\Tr(\chi_M\bH_\omega^2\chi_M)
&\leq
\int_{-M}^M \int_{-\infty}^\infty \frac{P_\omega(\log(\ee^x+\ee^y)^2}{(2\cosh\frac{x-y}{2})^2}\dd y\, \dd x
\\
&\leq
\int_{-M}^M \int_{-\infty}^\infty \frac{C_h^2}{(2\cosh\frac{x-y}{2})^2}\dd y\, \dd x
=2C_h^2 M, 
\end{align*}
and so the $\limsup$ is finite. 

Now suppose $\bH_\omega$ is positive and let us prove that the first moment of $\nu$ is finite. By an argument similar to the first part of the proof, the question reduces to the finiteness of 
\[
\limsup_{M\to\infty}\frac1{2M}\Tr(\chi_M\bH_\omega\chi_M).
\]
Since the kernel function $h_\omega$ is smooth and satisfies the estimate \eqref{eq:a1}, we can estimate the trace directly as follows: 
\begin{align*}
\Tr(\chi_M\bH_\omega\chi_M)
&=
\frac12\int_{-M}^M   P_\omega(\log(2\ee^x)) \dd x
\leq 
C_h M,
\end{align*}
which completes the proof. 
\end{proof}
This theorem is probably very far from being optimal. For example, for the Carleman operator all positive fractional moments of $\nu$ are finite.

\section{Szeg\H{o} type theorems}
\label{sec:szego}

\subsection{Preamble} 
In order to motivate what comes next, we recall that in the mathematical physics context, one has two equivalent approaches to defining the IDS measure of the Schr\"odinger operator \eqref{eq:Schrod}:
\begin{enumerate}[\rm (i)]
\item
through evaluating the normalised trace of the spectral projection of the Schr\"odinger operator, exactly as in \eqref{eq:c25} (see also \eqref{eq:c38} and \eqref{eq:c39});
\item
 through restricting the Schr\"odinger operator  onto the interval $(-M,M)$, with some boundary conditions, and identifying the IDS measure with the limit of the normalised eigenvalue counting measure of the restricted operator as $M\to\infty$. 
 \end{enumerate}
The equivalence of (i) and (ii) can be viewed as an analogue of the First Szeg\H{o} Limit theorem in this context. 

In this section we present two versions of Szeg\H{o} type theorems for Hankel operators. We have two theorems because there are two natural ways to restrict Hankel operators: 
\begin{enumerate}[\rm (a)]
\item
by restricting the integral kernel of $\bH_\omega$ to $(-M,M)$;
\item
for positive Hankel operators, by restricting the integration in \eqref{eq:a4omega} to $(-M,M)$.
\end{enumerate}
Approach (a) is the most direct analogue of the Szeg\H{o} type theorem for the Schr\"odinger operator referred to above. Approach (b) proves to be more useful in the context of positive Hankel operators. 
In terms of the pseudodifferential viewpoint \eqref{eq:a5} on positive Hankel operators, (a) and (b) correspond to restricting to $(-M,M)$  in either the ``coordinate'' or the ``momentum'' variable. 

\subsection{Approach (a)}

Here the restricted operator can be identified with $\chi_{M}\bH_\omega\chi_{M}$. 

\begin{theorem}\label{thm:c2a}
Let $\bH_\omega$ be a self-adjoint ergodic Hankel operator satisfying \eqref{eq:a1} with an $\omega$-independent $C_h$ and let $\nu$ be the IDS measure of $\bH_\omega$. Then almost surely, 
\begin{equation}
\lim_{M\to\infty}\frac1{2M}\Tr\varphi(\chi_{M}\bH_\omega\chi_{M})
=
\int_{-\infty}^\infty \varphi(\lambda)\dd\nu(\lambda)
\label{eq:c42}
\end{equation}
for any continuous function $\varphi$ compactly supported on $\bbR\setminus\{0\}$.
\end{theorem}
We note that the trace in the left-hand site of \eqref{eq:c42} exists by an argument similar to Lemma~\ref{lma:c1}. 
We will not prove this theorem (see, however, Remark~\ref{rmk:A2-4}) firstly because the proof is very similar to that of Theorem~\ref{thm:c2} below and secondly because various versions of this theorem are available in the literature, see e.g. Sections~4.B and 5.B of \cite{Pa-Fi:92}.  

This theorem has the following simple corollary. 
\begin{theorem}\label{thm:c2a2}
Let $\bH_\omega^{(1)}$ and $\bH_\omega^{(2)}$ be two ergodic Hankel operators as in Theorem~\ref{thm:c2a} such that 
\begin{equation}
0\leq \bH_\omega^{(1)}\leq \bH_\omega^{(2)}
\label{eq:c42a}
\end{equation}
almost surely. Let $\nu^{(1)}$ (resp. $\nu^{(2)}$) be the IDS measure of $\bH_\omega^{(1)}$ (resp. $\bH_\omega^{(2)}$). Then 
\begin{equation}
\nu^{(1)}((\lambda,\infty))\leq \nu^{(2)}((\lambda,\infty))
\label{eq:c42b}
\end{equation}
for all $\lambda>0$. 
\end{theorem}
\begin{proof}
We first note that in the context of Theorem~\ref{thm:c2a}, by a standard argument, approximating a step function by continuous functions, one obtains 
\begin{equation}
\lim_{M\to\infty}\frac1{2M}\Tr\chi_{(\lambda,\infty)}(\chi_{M}\bH_\omega\chi_{M})
=
\nu((\lambda,\infty))
\label{eq:x2}
\end{equation}
at any point $\lambda>0$ such that $\lambda$ is not a point mass of $\nu$. In particular, \eqref{eq:x2} holds true at a dense countable set of $\lambda>0$.

Next, we find from \eqref{eq:c42a}
\[
0\leq \chi_M \bH_\omega^{(1)}\chi_M\leq \chi_M \bH_\omega^{(2)}\chi_M,
\]
and therefore the counting functions of eigenvalues of these two operators satisfy
\[
\Tr\chi_{(\lambda,\infty)}(\chi_M \bH_\omega^{(1)}\chi_M)
\leq
\Tr\chi_{(\lambda,\infty)}(\chi_M \bH_\omega^{(2)}\chi_M)
\]
for any $\lambda>0$. Dividing by $2M$, taking $M\to\infty$ and using \eqref{eq:x2}, we obtain \eqref{eq:c42b} at a dense countable set of $\lambda>0$. A limiting argument yields \eqref{eq:c42b} for all $\lambda>0$.  
\end{proof}

\subsection{Approach (b)}

Let $\bH_\omega$ be a positive ergodic Hankel operator in $L^2(\bbR)$  with an ergodic family of measures $\Sigma_\omega$ as in Proposition~\ref{thm:b2}. Starting from the representation \eqref{eq:a4omega} for the integral kernel of $\bH_\omega$, we define the family of operators $\bH_\omega^{(M)}$ in $L^2(\bbR)$ with the integral kernels 
\begin{equation}
\bH_\omega^{(M)}(x,y)=\int_{-M}^M \beta(x-\xi)\beta(y-\xi)\dd\Sigma_\omega(\xi),
\label{eq:c1}
\end{equation}
where $M>0$ and $x,y\in\bbR$. Using the uniform local  boundedness condition \eqref{eq:a13}, it is straightforward to see that  $\bH_\omega^{(M)}$ is of trace class for any $M<\infty$. 

\begin{theorem}\label{thm:c2}
Let $\bH_\omega$ be a positive bounded ergodic Hankel operator and let $\nu$ be the corresponding IDS measure (see \eqref{eq:c38}, \eqref{eq:c39}). Then almost surely, 
\begin{align}
\lim_{M\to\infty}\frac1{2M}\Tr\varphi(\bH_\omega^{(M)})
&=\int_{-\infty}^\infty \varphi(\lambda)\dd\nu(\lambda)
\label{eq:c3}
\end{align}
for any continuous function $\varphi$ compactly supported on $\bbR\setminus\{0\}$.
\end{theorem}
In the rest of this section, we give a proof of this theorem. 
The proof introduces some technique that will come useful in the next section as well. 

\subsection{The strategy of the proof of Theorem~\ref{thm:c2}}
We start by rewriting the integral representations \eqref{eq:a4omega} and \eqref{eq:c1} in operator-theoretic terms, as operator factorisations. We will make use of the space $L^2(\bbR,\Sigma_\omega)$ (the $L^2$-space with the measure $\Sigma_\omega$). 
Let $\bL_\omega: L^2(\bbR)\to L^2(\bbR,\Sigma_\omega)$ be the convolution operator
\[
(\bL_\omega f)(\xi)=\int_{-\infty}^\infty \beta(x-\xi)f(x)\dd x, \quad \xi\in\bbR. 
\]
From the definition of $\beta$ and the uniform local boundedness condition \eqref{eq:a13} on $\Sigma_\omega$ it is easy to see that $\bL_\omega$ is bounded. The adjoint $\bL_\omega^*:L^2(\bbR,\Sigma)\to L^2(\bbR)$ is given by 
\[
(\bL_\omega^*f)(x)=\int_{-\infty}^\infty \beta(x-\xi)f(\xi)\dd\Sigma(\xi).
\]
In addition to the operator $\chi_M$ in $L^2(\bbR)$ (see \eqref{eq:cm}), we also need the operator of multiplication by the characteristic function of $(-M,M)$ acting in $L^2(\bbR,\Sigma_\omega)$. We denote the latter operator by $\chi_{M,\Sigma_\omega}$. With this notation, we have 
\begin{equation}
\bH_\omega=\bL_\omega^*\bL_\omega
\quad\text{ and }\quad
\bH_\omega^{(M)}=\bL_\omega^*\chi_{M,\Sigma_\omega}\bL_\omega.
\label{eq:c3b}
\end{equation}
We can now explain the strategy of the proof of Theorem~\ref{thm:c2}. 
We start from \eqref{eq:c28}; our aim is to replace $\chi_M\varphi(\bH_\omega)\chi_M$ in the left-hand side there by $\varphi(\bH_\omega^{(M)})$. This will be done in two steps, which in our new notation can be written as
\[
\chi_M\varphi(\bL_\omega^*\bL_\omega)\chi_M
\to
\varphi(\chi_M\bL_\omega^*\bL_\omega\chi_M)
\to
\varphi(\bL_\omega^*\chi_{M,\Sigma_\omega}\bL_\omega).
\]
Roughly speaking, this reduces to (i) preparing some commutator type estimates for $\bL_\omega$ and $\chi_M$, $\chi_{M,\Sigma_\omega}$ and (ii) using some standard functional calculus to account for the function $\varphi$. 

The argument below has a deterministic nature and works for every $\omega$, as long as the uniform local boundedness condition \eqref{eq:a13} is satisfied. 

\subsection{Commutator type estimates}
We introduce the shorthand 
\[
\chi_M^\perp=I-\chi_M \quad\text{ and }\quad \chi_{M,\Sigma_\omega}^\perp=I-\chi_{M,\Sigma_\omega}.
\]
Below $\norm{\cdot}_2$ is the Hilbert-Schmidt norm and $\norm{\cdot}_1$ is the trace norm. 
\begin{lemma}\label{lma:c8}
We have
\begin{align}
\norm{\chi_{M,\Sigma_\omega}\bL_\omega}_2&=O(\sqrt{M}),
\label{eq:c35}
\\
\norm{\chi_{M,\Sigma_\omega}^\perp \bL_\omega\chi_{M}}^2_2
+
\norm{\chi_{M,\Sigma_\omega}\bL_\omega\chi_M^\perp}_2^2
&=O(1)
\label{eq:c36}
\end{align}
as $M\to\infty$. 
\end{lemma}
\begin{proof}
Throughout the proof, we shall omit the index $\omega$ everywhere for readability. 
Let us prove \eqref{eq:c35}. 
Writing the Hilbert-Schmidt norm in terms of the integral kernel, we find
\begin{align*}
\norm{\chi_{M,\Sigma}\bL}_2^2
=
\int_{-M}^M \int_{-\infty}^\infty \beta(x-\xi)^2\dd x\, \dd\Sigma(\xi)
=
\int_{-M}^M\dd\Sigma(\xi) \int_{-\infty}^\infty \beta(x)^2\dd x
=
O(M)
\end{align*}
by the uniform local boundedness condition \eqref{eq:a13}. 

Consider \eqref{eq:c36}:
\begin{align}
\norm{\chi_{M,\Sigma}^\perp \bL\chi_{M}}^2_2
+
\norm{\chi_{M,\Sigma}\bL\chi_M^\perp}_2^2
=&
\int_{-M}^M\dd x\int_{\bbR\setminus(-M,M)}\dd\Sigma(\xi) \beta(x-\xi)^2
\notag
\\
&+
\int_{\bbR\setminus(-M,M)}\dd x\int_{-M}^M\dd\Sigma(\xi) \beta(x-\xi)^2. 
\label{eq:c18}
\end{align}
Using $\xi\leq \ee^\xi$ for $\xi\geq0$, we find that $\beta(\xi)\leq\ee^{-\abs{\xi}/2}$. Using the last inequality, we can estimate the first integral in the right-hand side of \eqref{eq:c18} as
\begin{align*}
\int_{-M}^M&\dd x\int_{\bbR\setminus(-M,M)}\dd\Sigma(\xi) \beta(x-\xi)^2
\\
&\leq
\int_{-\infty}^M \dd x\int_{M}^\infty \dd\Sigma(\xi) \ee^{-\abs{x-\xi}}
+
\int_{-M}^\infty \dd x\int_{-\infty}^{-M}\dd\Sigma(\xi) \ee^{-\abs{x-\xi}}
\\
&=
\int_{M}^\infty \ee^{-(\xi-M)}\dd\Sigma(\xi)
+
\int_{-\infty}^{-M}\ee^{\xi+M}\dd\Sigma(\xi)\, .
\end{align*}
By the uniform local boundedness condition \eqref{eq:a13} we find that the integrals in the right-hand side here are  bounded uniformly in $M$. In the same way one considers the second integral in the right-hand side of \eqref{eq:c18}. 
\end{proof}

\begin{lemma}\label{lma:A2-2}
We have
\begin{align}
 \norm{\chi_M \bL_\omega^*\bL_\omega \chi_M^\perp}_2&=O(1), 
\label{eq:A2-6}
\\
\norm{\chi_M \bL_\omega^*\bL_\omega\chi_M-\bL_\omega^*\chi_{M,\Sigma_\omega}\bL_\omega}_1&=O(\sqrt{M}), 
\label{eq:A2-5}
\end{align}
as $M\to\infty$. 
\end{lemma}
\begin{proof}
Again, we omit the index $\omega$ throughout the proof. 
Let us prove \eqref{eq:A2-6}. Writing $I=\chi_{M,\Sigma}+\chi_{M,\Sigma}^\perp$, we find
\[
\chi_M  \bL^* \bL \chi_M^\perp
=
\chi_M  \bL^*\chi_{M,\Sigma} \bL \chi_M^\perp
+
\chi_M  \bL^*\chi_{M,\Sigma}^\perp  \bL \chi_M^\perp.
\]
Using the triangle inequality for the Hilbert-Schmidt norm and the inequality $\norm{AB}_2\leq\norm{A}\norm{B}_2$, we find 
\[
\norm{\chi_M  \bL^* \bL \chi_M^\perp}_2
\leq
\norm{ \bL}\norm{\chi_{M,\Sigma} \bL \chi_M^\perp}_2
+
\norm{ \bL}\norm{\chi_{M,\Sigma}^\perp  \bL\chi_M}_2
=O(1)
\]
where we have used \eqref{eq:c36} at the last step. 

Let us prove \eqref{eq:A2-5}. 
Writing $I=\chi_M+\chi_M^\perp$ and $I=\chi_{M,\Sigma}+\chi_{M,\Sigma}^\perp$, after a little algebra we find 
\begin{align*}
\chi_M \bL^*\bL\chi_M-\bL^*\chi_{M,\Sigma}\bL
=&
\chi_M \bL^*\chi_{M,\Sigma}^\perp \bL\chi_M
-\chi_M^\perp \bL^*\chi_{M,\Sigma}\bL\chi_M^\perp
\\
&-\chi_M \bL^*\chi_{M,\Sigma}\bL\chi_M^\perp
-\chi_M^\perp \bL^*\chi_{M,\Sigma}\bL\chi_M.
\end{align*}
Using the triangle inequality for the trace norm and the ``Cauchy-Schwartz inequality in Schatten classes'' $\norm{AB}_1\leq\norm{A}_2\norm{B}_2$, from here we find
\begin{multline*}
\norm{\chi_M \bL^*\bL\chi_M-\bL^*\chi_{M,\Sigma} \bL}_1
\\
\leq
\norm{\chi_{M,\Sigma}^\perp \bL\chi_M}_2^2
+
\norm{\chi_{M,\Sigma}\bL\chi_M^\perp}_2^2
+
2\norm{\chi_{M,\Sigma}\bL}_2\norm{\chi_{M,\Sigma}\bL\chi_M^\perp}_2.
\end{multline*}
Using Lemma~\ref{lma:c8} in the right-hand side, we obtain \eqref{eq:A2-5}. 
\end{proof}

\subsection{Some functional calculus}
\begin{lemma}\label{lma:A2-3}
Let $\varphi\in C^\infty(\bbR)$ with $\supp\varphi\subset\bbR\setminus\{0\}$. Then 
\begin{align}
\abs{\Tr(\chi_M\varphi(\bH_\omega)\chi_M)-\Tr\varphi(\chi_M\bH_\omega\chi_M)}
&=O(1),
\label{eq:A2-7}
\\
\abs{\Tr\varphi(\chi_M\bH_\omega\chi_M)-\Tr\varphi(\bH_\omega^{(M)})}
&=O(\sqrt{M}),
\label{eq:A2-8}
\end{align}
as $M\to\infty$. 
\end{lemma}
\begin{proof}
In order to prove \eqref{eq:A2-7}, we use the inequality from \cite[Theorem~1.2]{La-Sa}:
\begin{equation}
\abs{\Tr(\chi_M\varphi(\bH_\omega)\chi_M)-\Tr\varphi(\chi_M\bH_\omega\chi_M)}
\leq
\frac12\norm{\varphi''}_{L^\infty}\norm{\chi_M \bH_\omega \chi_M^\perp}_2^2
\label{eq:A2-10}
\end{equation}
and invoke \eqref{eq:A2-6} (recall that $\bH_\omega=\bL_\omega^*\bL_\omega$). 

Let us prove \eqref{eq:A2-8}. Let us denote $A=\chi_M\bH_\omega\chi_M$, $B=\bH^{(M)}_\omega$ and estimate the trace by the trace norm, 
\[
\abs{\Tr(\varphi(A)-\varphi(B))}
\leq
\norm{\varphi(A)-\varphi(B)}_1,
\]
and use the inequality \cite{Peller2}
\begin{equation}
\norm{\varphi(A)-\varphi(B)}_1\leq C(\varphi)\norm{A-B}_1,
\label{eq:A2-9}
\end{equation}
where $C(\varphi)$ is (up to a multiplicative constant) the norm of $\varphi$ in the Besov class $B^1_{\infty,1}$. Using \eqref{eq:A2-5} in the right-hand side of \eqref{eq:A2-9}, we arrive at \eqref{eq:A2-8}. 
\end{proof}

\subsection{Proof of Theorem~\ref{thm:c2}}
 Combining \eqref{eq:A2-7} and \eqref{eq:A2-8}, we find 
\[
\abs{\Tr(\chi_M\varphi(\bH_\omega)\chi_M)-\Tr\varphi(\bH_\omega^{(M)})}
=O(\sqrt{M})
\]
as $M\to\infty$. Using Proposition~\ref{thm:c5}, from here we obtain the required relation \eqref{eq:c3} for any $\varphi\in C^\infty$ vanishing near the origin. Extension to all continuous $\varphi$ is a standard approximation argument. The proof of Theorem~\ref{thm:c2} is complete. \qed

\begin{remark}\label{rmk:A2-4}
Relation \eqref{eq:A2-7} provides the proof of Theorem~\ref{thm:c2a} in the case of positive Hankel operators. 
\end{remark}

\section{The IDS measure for positive Hankel operators}
\label{sec.cc}

Throughout this section, $\bH_\omega$ is an ergodic a.s. bounded and positive Hankel operator in $L^2(\bbR)$ with the IDS measure $\nu$, and $\Sigma_\omega$ is the associated measure in \eqref{eq:a4omega}.

\subsection{The continuity of $\nu$ for positive Hankel operators}

\begin{theorem}
Let $\bH_\omega$ be a  positive bounded ergodic Hankel operator with the IDS measure $\nu$ (see \eqref{eq:c38}, \eqref{eq:c39}). Then $\nu$ has no point masses. 
\end{theorem}

\begin{proof}
We recall that by Corollary~\ref{crl:MPT}, for all $\omega$, all positive eigenvalues of $\bH_\omega$ are simple. Take any $\lambda_*>0$ and any $\omega$. Let us apply \eqref{eq:c25} with $\Delta=\{\lambda_*\}$. The integral kernel in \eqref{eq:c25} is the kernel of a rank one operator (or zero). From here it is clear that the limit in \eqref{eq:c25} vanishes, and so $\nu(\{\lambda_*\})=0$.
\end{proof}

In Section~\ref{sec:d6} we discuss an example demonstrating that the conclusion of the above theorem fails without the assumption of positivity. 

For the purposes of comparison, we recall the continuity of the IDS measure of an ergodic Schr\"odinger operator on $L^2(\bbR^d)$ \cite[Theorem~3.2]{Pa-Fi:92} as well as its discrete analogue on $\ell^2(\bbZ^d)$ \cite[Theorem~3.4]{Pa-Fi:92} for any $d\geq1$. 

\subsection{The density of the support of $\Sigma_\omega$}
The key result of this section is Theorem~\ref{thm:D1} below, which identifies the total mass $\nu(\bbR_+)$ with the density of the support of $\Sigma_\omega$. We start by defining this density:
\[
D(\supp\Sigma_\omega)=\frac1{\tau}\bE\{\#\bigl((\supp\Sigma_\omega)\cap[0,\tau)\bigr)\},
\]
finite or infinite. 
(Notation $D(\supp\Sigma_\omega)$ is not perfect as in fact this density is independent of $\omega$.) The following statement is a simple version of an Ergodic Theorem. 

\begin{lemma}\label{lma:cc4}
 Let $\Sigma_\omega$ be an ergodic measure as in \eqref{eq:add1}  with the density $D(\supp\Sigma_\omega)$ (finite or infinite). Then almost surely, 
\begin{equation}
\lim_{M\to\infty}\frac1{2M}\#(\supp\Sigma_\omega\cap(-M,M))=D(\supp\Sigma_\omega). 
\label{eq:cc6}
\end{equation}
\end{lemma}
\begin{proof}
First let us assume that $D(\supp\Sigma_\omega)$ is finite. Consider the function
\[
F(\omega)=\frac1\tau\#(\supp\Sigma_\omega\cap[0,\tau)).
\]
Observe that by the ergodicity condition \eqref{eq:add1}, 
\begin{equation}
F(T_a\omega)=\frac1\tau\#(\supp\Sigma_\omega\cap[a,a+\tau)).
\label{eq:add4b}
\end{equation}
By the Ergodic Theorem (Theorem~\ref{thm:erg1}), we have almost surely 
\[
D(\supp\Sigma_\omega)
=\bE\{F(\omega)\}
=\lim_{M\to\infty}\frac1{2M}\int_{-M}^M F(T_a\omega)\dd a
\]
in the continuous case and 
\[
D(\supp\Sigma_\omega)
=\bE\{F(\omega)\}
=\lim_{N\to\infty}\frac1{2N+1}\sum_{n=-N}^N F(T^n\omega)
\]
in the discrete case. By \eqref{eq:add4b}, it is easy to identify the right-hand sides here with the limit in \eqref{eq:cc6}. 

Let us briefly explain how to modify this argument if $D(\supp\Sigma_\omega)$ is infinite. 
Let $F$ be as above and for a rational $R>0$, define 
\[
F^{(R)}(\omega)=\min\{R,F(\omega)\}.
\]
Consider the continuous case. Applying the previous argument to $F^{(R)}$, 
\[
\liminf_{M\to\infty}\frac1{2M}
\#(\supp\Sigma_\omega\cap(-M,M))
\geq
\liminf_{M\to\infty}\frac1{2M}\int_{-M}^M F^{(R)}(T_a\omega)\dd a
=
\bE\{F^{(R)}(\omega)\}
\]
almost surely. Since  $D(\supp\Sigma_\omega)$ is infinite, the right-hand side here tends to infinity as $R\to\infty$. This proves that the left-hand side is infinite almost surely. 
The proof in the discrete case is similar, with sums instead of integrals.
\end{proof}

We are now ready to state the main result of this section. 
\begin{theorem}\label{thm:D1}
Let $\bH_\omega$ be an almost surely positive and bounded ergodic Hankel operator in $L^2(\bbR)$ with the IDS measure $\nu$ and let $\Sigma_\omega$ be the associated measure in \eqref{eq:a4omega} with the density $D(\supp\Sigma_\omega)$, finite or infinite. Then 
\[
\boxed{\nu(\bbR_+)=D(\supp\Sigma_\omega)\, .}
\]
\end{theorem}
We give the proof in Sections~\ref{sec:cc3}-\ref{sec:cc6} and discuss some consequences in Sections~\ref{sec:c7}-\ref{sec:c8}.

\subsection{Reduction to an operator in $L^2(\bbR,\Sigma_\omega)$}
\label{sec:cc3}

By Theorem~\ref{thm:c2}, we have 
\begin{align}
\lim_{M\to\infty}\frac1{2M}\Tr\varphi(\bH_\omega^{(M)})
&=\int_{-\infty}^\infty \varphi(\lambda)\dd\nu(\lambda)
\label{eq:add4}
\end{align}
for any $\varphi$ continuous and compactly supported on $\bbR_+$. The idea of the proof is to use the factorisation 
\[
\bH^{(M)}_\omega=\bL_\omega^*\chi_{M,\Sigma_\omega}\bL_\omega
\]
as in \eqref{eq:c3b}, and to ``swap the factors'' here. 
We define the operator $\bG_\omega$ in $L^2(\bbR,\Sigma_\omega)$ by 
\[
\bG_\omega=\bL_\omega \bL_\omega^*.
\]
We recall that for any bounded operator $X$ in a Hilbert space, the operators $X^*X$ and $XX^*$ are unitarily equivalent modulo kernels, i.e. 
\[
X^*X|_{(\Ker X)^\perp}
\quad\text{ is unitarily equivalent to }\quad 
XX^*|_{(\Ker X^*)^\perp}.
\]
Thus, the operators $\bH_\omega$ in $L^2(\bbR)$ and $\bG_\omega$ in $L^2(\bbR,\Sigma_\omega)$ are unitarily equivalent modulo kernels. 
Similarly, $\bH_\omega^{(M)}$ is unitarily equivalent to 
\[
\bG_\omega^{(M)}=\chi_{M,\Sigma_\omega}\bG_\omega \chi_{M,\Sigma_\omega}
\quad 
\text{ in $L^2(\bbR,\Sigma_\omega)$}
\]
modulo kernels. From here and \eqref{eq:add4} we find 
\begin{align}
\lim_{M\to\infty}\frac1{2M}\Tr\varphi(\bG_\omega^{(M)})
&=\int_{-\infty}^\infty \varphi(\lambda)\dd\nu(\lambda)
\label{eq:add4a}
\end{align}
for any $\varphi$ continuous and compactly supported on $\bbR_+$. This characterisation of $\nu$ will be our starting point in the proof of Theorem~\ref{thm:D1}. 

\subsection{$\bG_\omega$ has a trivial kernel}
We will need a simple lemma on $\bG_\omega$. 
\begin{lemma}\label{lma:add2}
The kernel of $\bG_\omega$ in $L^2(\bbR,\Sigma_\omega)$ is trivial for all $\omega$. 
\end{lemma}
\begin{proof}
This statement has a deterministic nature, and so we drop the index $\omega$. In fact, it was proved in \cite[Theorem~3.1(iv)]{PuTreil1}, but for convenience we reproduce the proof here.

Suppose $\bG F=0$ for some $F\in L^2(\bbR)$. Evaluating the quadratic form of $\bG=\bL \bL^*$ on $F$,  we find $\bL^*F=0$, or explicitly 
\begin{equation}
\int_{-\infty}^\infty \beta(x-\xi)F(\xi)\dd\Sigma(\xi)=0
\label{eq:add7}
\end{equation}
for a.e. $x\in\bbR$. We need to prove that $F=0$. The quickest way to do this is to relate the question to a classical uniqueness theorem for the Laplace transform. Let us make the change of variable $t=\ee^{-\xi}$ and rewrite \eqref{eq:add7} in terms of the measure $\sigma$, where the measures $\sigma$ and $\Sigma$ are related as explained in Section~\ref{sec:aa3}. After a short calculation, we find 
\begin{equation}
\int_0^\infty \ee^{-ts}f(t)\dd\sigma(t)=0,\quad \text{ a.e. $s\in\bbR_+$}
\label{eq:add8}
\end{equation}
where $f(\ee^{-\xi})=\ee^{-\xi/2}F(\xi)$, $f\in L^2(\bbR_+,\sigma)$. The integral here is the Laplace transform of $f$. We note that the Laplace transform of $f$ is an analytic function in the right half-plane, and \eqref{eq:add8} shows that it vanishes identically there. Now we can refer to a classical uniqueness theorem on Laplace transforms (see e.g. \cite[Chapter II, Theorem 6.3]{Widder}) to conclude that $f(t)=0$ for $\sigma$-a.e. $t>0$. 
\end{proof}

\subsection{Another Szeg\H{o} type theorem}
\begin{theorem}
For any $C^2$-smooth compactly supported function $\varphi$ on $\bbR_+$ we have almost surely
\[
\Tr\varphi(\bG_\omega^{(M)})-\Tr\bigl(\chi_{M,\Sigma_\omega}\varphi(\bG_\omega)\chi_{M,\Sigma_\omega}\bigr)=O(1), \quad M\to\infty.
\]
\end{theorem}
\begin{proof}
By the inequality \eqref{eq:A2-10} (with $\bG_\omega$, $\chi_{M,\Sigma_\omega}$ in place of $\bH_\omega$, $\chi_M$), it suffices to establish the estimate 
\[
\norm{\chi_{M,\Sigma_\omega}\bG_\omega\chi_{M,\Sigma_\omega}^\perp}_2=O(1),
\]
where $\chi_{M,\Sigma_\omega}^\perp=1-\chi_{M,\Sigma_\omega}$. Since $\bG_\omega=\bL_\omega \bL_\omega^*$, we have 
\begin{align*}
\chi_{M,\Sigma_\omega}\bG_\omega \chi_{M,\Sigma_\omega}^\perp
=&
(\chi_{M,\Sigma_\omega}\bL_\omega \chi_{M,\Sigma_\omega}^\perp)\bL_\omega^* \chi_{M,\Sigma_\omega}^\perp
\\
&+
\chi_{M,\Sigma_\omega}\bL_\omega (\chi_{M,\Sigma_\omega}\bL_\omega^*\chi_{M,\Sigma_\omega}^\perp),
\end{align*}
and by Lemma~\ref{lma:A2-2} the Hilbert-Schmidt norm of the right-hand side here is $O(1)$ as $M\to\infty$. 
\end{proof}

\subsection{Proof of Theorem~\ref{thm:D1}}
\label{sec:cc6}
By \eqref{eq:add4a} and the previous theorem, we have 
\begin{equation}
\lim_{M\to\infty}\frac1{2M}\Tr(\chi_{M,\Sigma_\omega}\varphi(\bG_\omega)\chi_{M,\Sigma_\omega})=\int_0^\infty\varphi(\lambda)\dd\nu(\lambda),
\label{eq:add5}
\end{equation}
for any $C^2$-smooth compactly supported function $\varphi$ on $\bbR_+$. Fix such $\varphi$ and denote 
\[
F(\omega)=\frac1\tau\Tr\bigl(\chi_{[0,\tau),\Sigma_\omega}\varphi(\bG_\omega)\chi_{[0,\tau),\Sigma_\omega}\bigr).
\]
Here we use the argument similar to the proof of Lemma~\ref{lma:cc4}, except that now the expectation of $F$ is finite, because $\varphi$ is compactly supported. By the Ergodic Theorem (Theorem~\ref{thm:erg1}), we have 
\[
\lim_{N\to\infty}\frac1{2N+1}\sum_{n=-N}^N F(T^n\omega)=\bE\{F(\omega)\}
\]
in the discrete case and 
\[
\lim_{M\to\infty}\frac1{2M}\int_{-M}^M F(T_a\omega)\dd a=\bE\{F(\omega)\}
\]
in the continuous case. In both cases, the limits in the left-hand side are easy to identify with the limit in the left-hand side of \eqref{eq:add5}, which yields
\[
\frac1\tau\bE\{\Tr(\chi_{[0,\tau),\Sigma_\omega}\varphi(\bG_\omega)\chi_{[0,\tau),\Sigma_\omega})\}
=
\int_0^\infty\varphi(\lambda)\dd\nu(\lambda).
\]
Let us write the last relation for a sequence of functions $\varphi=\varphi_n$, where each $\varphi_n$ is $C^2$-smooth, compactly supported on $\bbR_+$ and satisfies $0\leq \varphi_n\leq1$. Moreover, we assume that $\varphi_n(\lambda)\to1$ as $n\to\infty$ for each $\lambda>0$ and $\varphi_n(\lambda)$ is monotonically increasing:
\begin{equation}
\frac1\tau\bE\{\Tr(\chi_{[0,\tau),\Sigma_\omega}\varphi_n(\bG_\omega)\chi_{[0,\tau),\Sigma_\omega})\}
=
\int_0^\infty\varphi_n(\lambda)\dd\nu(\lambda).
\label{eq:add6}
\end{equation}
It is clear that the right-hand side of \eqref{eq:add6} converges to $\nu(\bbR_+)$ (which may be finite or infinite) as $n\to\infty$. Consider the left-hand side. By Lemma~\ref{lma:add2}, for each $\omega$ we have the convergence $\varphi_n(\bG_\omega)\to I$ in the strong operator topology as $n\to\infty$. It follows that 
\[
\chi_{[0,\tau),\Sigma_\omega}\varphi_n(\bG_\omega)\chi_{[0,\tau),\Sigma_\omega}
\to 
\chi_{[0,\tau),\Sigma_\omega}
\]
in the strong operator topology for each $\omega$. Thus, the trace of the operator in the left-hand side here converges to the dimension of the space $L^2([0,\tau),\Sigma_\omega)$ (finite or infinite): 
\[
\Tr(\chi_{[0,\tau),\Sigma_\omega}\varphi_n(\bG_\omega)\chi_{[0,\tau),\Sigma_\omega})
\to
\dim L^2([0,\tau),\Sigma_\omega)
=
\#\bigl((\supp\Sigma_\omega)\cap[0,\tau)\bigr)
\]
as $n\to\infty$ for each $\omega$. Moreover, by the monotonicity of $\varphi_n$, the sequence in the left-hand side here is monotone increasing in $n$. By the monotone convergence theorem, we can pass to the limit
in the left-hand side of \eqref{eq:add6}, which finally yields
\[
\frac1\tau
\bE\{\#\bigl((\supp\Sigma_\omega)\cap[0,\tau)\bigr)\}=\nu(\bbR_+),
\]
and both sides may be finite or infinite. Finally, we recall that the left-hand side here is the definition of the density $D(\supp\Sigma_\omega)$. The proof of Theorem~\ref{thm:D1} is complete. \qed

\subsection{The case of a trivial kernel}
\label{sec:c7}

\begin{theorem}
Let $\bH_\omega$ be an ergodic a.s. bounded and positive Hankel operator in $L^2(\bbR)$ such that $\Ker\bH_\omega=\{0\}$ almost surely. Then $\nu(\bbR_+)=\infty$. 
\end{theorem}
\begin{proof}
By Theorem~\ref{thm:D1}, it suffices to prove that $\Ker\bH_\omega=\{0\}$ implies $D(\supp\Sigma_\omega)=\infty$. Let us prove the contrapositive statement: if $D(\supp\Sigma_\omega)<\infty$, then condition \eqref{eq:b9a} is satisfied and so $\dim\Ker\bH_\omega=\infty$. 
In the discrete case, using the elementary inequality 
\[
\sech(\xi+t)\leq \ee^{\abs{t}}\sech\xi, \quad \xi, t\in\bbR,
\]
and the ergodicity condition \eqref{eq:add1} for $\Sigma_\omega$, we find
\begin{align}
\sum_{\xi\in\supp\Sigma_\omega}\sech\xi
&\leq
\ee^{\tau}\sum_{n=-\infty}^\infty(\sech\tau n) \#\bigl(\supp\Sigma_\omega\cap[\tau n, \tau(n+1))\bigr)
\notag
\\
&=
\ee^{\tau}\sum_{n=-\infty}^\infty(\sech\tau n) F(T^n\omega),
\quad 
F(\omega)=\#(\supp\Sigma_\omega\cap[0,\tau)).
\label{eq:add2}
\end{align}
Since by assumption the density $D(\supp\Sigma_\omega)$ is finite, 
by the Ergodic Theorem (Theorem~\ref{thm:erg1}), the limit 
\[
\lim_{N\to\infty}\frac1{2N+1}\sum_{n=-N}^N F(T^n\omega)
\]
is finite almost surely. From here, using summation by parts, it is straightforward to see that the series in the right-hand side of \eqref{eq:add2} is finite almost surely. 

Similarly, in the continuous case, with the same $F$ we have 
\begin{align}
\sum_{\xi\in\supp\Sigma_\omega}\sech\xi
&\leq
\frac{\ee^\tau}{\tau}\int_{-\infty}^\infty (\sech\xi)\#\bigl(\supp\Sigma_\omega\cap[\xi,\xi+\tau)\bigr)\dd\xi
\notag
\\
&=\frac{\ee^\tau}{\tau}\int_{-\infty}^\infty (\sech\xi)F(T_\xi\omega)\dd\xi.
\label{eq:add3}
\end{align}
By the ergodic theorem, the limit 
\[
\lim_{M\to\infty}
\frac1{2M}\int_{-M}^M F(T_\xi\omega)\dd\xi
\]
is finite. Integrating by parts, we find that the integral in the right-hand side of \eqref{eq:add3} is finite. The proof is complete. 
\end{proof}

We can summarise the last theorem as follows:
\begin{equation}
\nu(\bbR_+)+\dim\Ker\bH_\omega=\infty
\label{eq:xx1}
\end{equation}
for positive Hankel operators.

Suppose the kernel of $\bH_\omega$ is trivial. 
Since $\nu(\bbR_+)=\infty$, it is reasonable to ask about the behaviour of $\nu((\lambda,\infty))$ as $\lambda\to0_+$. We give the simplest statement of this nature. Consider  $\bH_\omega$ such that the corresponding measure $\dd\Sigma_\omega$ is purely absolutely continuous, 
\begin{equation}
\dd\Sigma_\omega(\xi)=\rho_\omega(\xi)\dd\xi,
\label{eq:cc1}
\end{equation}
where the density $\rho_\omega$ satisfies, almost surely,
\begin{equation}
0<\rho_{\min}=\essinf_{\xi\in\bbR}\rho_\omega(\xi)\quad\text{ and }\quad
\esssup_{\xi\in\bbR}\rho_\omega(\xi)=\rho_{\max}<\infty.
\label{eq:cc2}
\end{equation}
We recall that in this case, by Proposition~\ref{prop:b8}, the kernel of $\bH_\omega$ is trivial almost surely. 
\begin{theorem}\label{thm:cc8}
Let $\bH_\omega$ be a positive ergodic Hankel operator in $L^2(\bbR)$  satisfying \eqref{eq:cc1} and \eqref{eq:cc2}. Then the corresponding IDS measure $\nu$ satisfies
\begin{equation}
\nu((\lambda,\infty))=\frac{2}{\pi}\log\frac1{\lambda}+O(1) \quad\text{ as $\lambda\to0_+$.}
\label{eq:cc3}
\end{equation}
\end{theorem}
\begin{proof}
By \eqref{eq:cc2}, we have, almost surely,
\[
\rho_{\min}\bH_C\leq \bH_\omega\leq \rho_{\max} \bH_C,
\]
where $\bH_C=\calE H_C\calE^*$ and $H_C$ is the Carleman operator. From here  by Theorem~\ref{thm:c2a2} we obtain 
\[
\nu_C((\lambda/\rho_{\min},\infty))
\leq
\nu((\lambda,\infty))
\leq
\nu_C((\lambda/\rho_{\max},\infty)).
\]
Using \eqref{eq:cc5} yields the required asymptotics \eqref{eq:cc3}. 
\end{proof}
We observe that the asymptotics \eqref{eq:cc3} is independent of $\rho_\omega$ within the class \eqref{eq:cc2}. This can be viewed as an analogue of the high-energy asymptotics of the IDS of the Schr\"odinger operator, see e.g. \cite[formula (5.36)]{Pa-Fi:92}.

\subsection{The case of infinite-dimensional kernel}\label{sec:c8}
Suppose $\dim\Ker\bH_\omega=\infty$, i.e. condition \eqref{eq:b9a} for $\Sigma_\omega$ is satisfied. It is easy to construct many examples when $D(\supp\Sigma_\omega)$ is finite, e.g. periodic $\Sigma_\omega$ (see Section~\ref{sec.dd} below). It is reasonable to ask whether $D(\supp\Sigma_\omega)$ is \emph{always} finite in the case of infinite dimensional kernels.  Here we show that this is false. We explain how to construct an example of measures $\Sigma_\omega$  satisfying condition \eqref{eq:b9a} with $D(\supp\Sigma_\omega)=\infty$.

Let $\{x_\omega(m)\}_{m\in\bbZ}$ be the sequence of i.i.d. variables with the following common distribution: $x_\omega(0)=2^k$ with probability $2^{-k}$ for all integers $k\geq1$. It is clear that $\bE\{x_\omega(0)\}=\infty$; this is the key feature of this example. 

For every $\omega$, let us construct the measure $\Sigma_\omega$ as follows:
\[
\Sigma_\omega=\sum_{m=-\infty}^\infty \frac{1}{x_\omega(m)}\sum_{j=0}^{x_\omega(m)-1}\delta_{m+(j/x_\omega(m))}. 
\]
In other words, for every $m$,  the measure $\Sigma_\omega$ is supported on $x_\omega(m)$ equidistant points on the interval $[m,m+1)$. This measure satisfies the ergodicity condition \eqref{eq:add1} with $\tau=1$ and the uniform local boundedness condition \eqref{eq:a13}. 

We have 
\begin{equation}
\sum_{\xi\in\supp\Sigma_\omega} \sech\xi \leq C\sum_{m=-\infty}^\infty (\sech m) x_\omega(m).
\label{eq:exa}
\end{equation}
Using Kolmogorov's three-series theorem, it is easy to see that the series on the right-hand side of \eqref{eq:exa} converges almost surely. Thus, \eqref{eq:b9a} is satisfied almost surely.

On the other hand, by the arguments similar to the proof of Lemma~\ref{lma:cc4} (and based on the Ergodic Theorem), we have
\[
D(\supp\Sigma_\omega)=\lim_{N\to\infty}\frac1{2N}\sum_{m=-N}^N x_\omega(m)=\bE\{x_\omega(0)\}=\infty.
\]

\section{Periodic Hankel operators}
\label{sec.d}

\subsection{The main result: properties of $\nu$}
In this section we discuss the IDS measure for periodic Hankel operators $\bH_\omega$ introduced in Example~\ref{ex:b4}. We follow \cite{PuSobolev}, where it was assumed that the $\tau$-periodic function $P$ in \eqref{eq:per2} is sufficiently smooth. This smoothness was expressed in terms of the Fourier coefficients of $P$, viz. 
\begin{equation}
\widetilde P_m=\frac1\tau\int_0^\tau \ee^{-\ii m\frac{2\pi}{\tau} \xi}P(\xi)\dd\xi, \quad m\in\bbZ.
\label{eq:d6}
\end{equation}
The periodic Hankel operators $\bH_\omega$ satisfying 
\begin{equation}
\sum_{m=-\infty}^\infty |\widetilde P_m|\, (1+\abs{m})^{1/2} < \infty,
\label{eq:d4}
\end{equation}
were dubbed \emph{smooth periodic Hankel operators} in \cite{PuSobolev}. In particular, this condition implies the continuity of $P$. 

The following theorem characterises the properties of $\nu$. 
Again, for positive Hankel operators we see a similarity with Schr\"odinger operators \eqref{eq:Schrod}.

\begin{theorem}\label{thm:d-n1}
Let $\bH_\omega$ be a smooth $\tau$-periodic Hankel operator, and let $\nu$ be its IDS measure. 
\begin{enumerate}[\rm (i)] 
\item
The measure $\nu$ has no singular continuous component. Moreover, if $\bH_\omega$ is positive, then $\nu$ has no pure point (p.p.) part, i.e. $\nu$ is purely absolutely continuous (a.c.). 
\item
If the a.c. component of $\nu$ is non-trivial, then it can be represented as a finite or infinite sum
\begin{equation}
\nu^{\rm ac}=\frac1{\tau}\sum_n \nu_n,
\label{eq:d-n2}
\end{equation}
where each $\nu_n$ is a purely a.c. probability measure and $\sigma_n=\supp \nu_n$ is a closed bounded interval (we will call it a \emph{spectral band}) separated away from the origin. The spectral bands are intervals of the a.c. spectrum of $\bH_\omega$. 
They satisfy:
\begin{itemize}
\item
$\sigma_n\cap \sigma_m$ is either empty or a single point for any $n\not=m$;
\item
$\sigma_n\cap(-\sigma_m)=\varnothing$ for any $n,m$;
\item
if there are infinitely many spectral bands, then the origin is their one and only accumulation point. 
\end{itemize}
\item
If the p.p. component of $\nu$ is non-trivial, then it has the structure
\begin{equation}
\nu^{\rm pp}=\frac1{\tau}\sum_n (\delta_{\lambda_n}+\delta_{-\lambda_n}),
\label{eq:d-n1}
\end{equation}
where $\{\lambda_n\}$ is a finite or infinite set of positive numbers. If this set is infinite, then $\lambda_n\to0$ as $n\to\infty$. The points $\{-\lambda_n\}\cup\{\lambda_n\}$ are the eigenvalues of infinite multiplicity of $\bH_\omega$; we will call them \emph{flat spectral bands.}
\end{enumerate}
\end{theorem}

\begin{remark}
\begin{enumerate}[1.]
\item
In the next section, we will give an example of an operator with a non-trivial $\nu^{\rm ac}$ and another example with a non-trivial $\nu^{\rm pp}$. It is a little harder to give an example where both a.c. and p.p. parts of $\nu$ are non-trivial (i.e. flat and non-flat bands coexist); an example of this kind is furnished by the \emph{Mathieu--Hankel} operator, see \cite[Section 8]{PuSobolev}. 
\item
We do not know if the supports of $\nu^{\rm pp}$ and $\nu^{\rm ac}$ can have a non-empty intersection (i.e. if a flat band can be located on top of a non-flat band). 
\item
From \eqref{eq:d-n1} we see that $\nu^{\rm pp}$ is symmetric with respect to the reflection around zero. This is consistent with the fact that for positive Hankel operators, the pure point component of $\nu$ disappears. 
\item
If two spectral bands $\sigma_n$ and $\sigma_m$ have a common point (i.e. if they \emph{touch}), then in principle they can be ``merged'' to form a larger spectral band. It is, however, useful to think of them as two separate bands, because they correspond to different band functions, see Theorem~\ref{thm.branches}.
\item
From the more detailed description of the band functions below, it is possible to get more refined information about the a.c. measures $\nu_n$. For example, the density of $\nu_n$ is a smooth function inside the spectral band $\sigma_n$; this density may have a quadratic singularity at the endpoints of $\sigma_n$. 
\item
From \eqref{eq:d-n2} and \eqref{eq:d-n1} we obtain a simple case of a ``gap labelling theorem'': in each spectral gap, the IDS
\[
\int_\lambda^\infty \dd\nu(x)
\]
is constant of the form $N/\tau$, where $N$ is an integer equal to the number of bands (both flat and non-flat) in the interval $(\lambda,\infty)$. See \cite[Section~16.A]{Pa-Fi:92} and references therein for the gap labelling theorem for almost periodic operators. 

\end{enumerate}
\end{remark}

Essentially, Theorem~\ref{thm:d-n1} follows from the construction of \cite{PuSobolev}. However, the notion of the IDS measure was not considered in \cite{PuSobolev}. In the rest of this section we recall the key points from \cite{PuSobolev} and show how to put them together with our definition of the IDS measure to yield the proof of Theorem~\ref{thm:d-n1}.

\subsection{The Floquet--Bloch decomposition}
In this subsection, we recall the standard Floquet--Bloch decomposition for \emph{general} periodic operators in $L^2(\bbR)$, i.e. operators commuting with shifts. Let $A$ be a bounded self-adjoint operator in $L^2(\bbR)$ commuting with $U_\tau$ for a fixed $\tau>0$.  Let $\ell^2=\ell^2(\bbZ)$ and let us consider the constant-fiber direct integral of Hilbert spaces:
\[
\mathfrak H=\int_{(-\pi/\tau,\pi/\tau)}^\oplus \ell^2\,  \dd k;
\]
here $(-\pi/\tau,\pi/\tau)$ is the dual period cell. In other words, $\mathfrak H$ is the $L^2$-space of $\ell^2$-valued functions $F=F(k)$ on $(-\pi/\tau,\pi/\tau)$ with respect to the Lebesgue measure.
Clearly, $\mathfrak H$ is isomorphic to $L^2(\bbR)$ in a natural way: a function $f\in L^2(\bbR)$ corresponds to 
\[
F(k)=\{f(k+\tfrac{2\pi}{\tau} n)\}_{n\in\bbZ}, \quad k\in (-\pi/\tau,\pi/\tau).
\]

We define the unitary map $\calU:L^2(\bbR)\to\mathfrak H$ by 
\[
[(\calU f)(k)]_n=\widehat f(k+\tfrac{2\pi}{\tau} n)=\frac1{\sqrt{2\pi}}\int_{-\infty}^\infty f(x)\ee^{-\ii kx-\ii \frac{2\pi}{\tau} nx}\dd x.
\]
The map $\calU$ is the combination of the Fourier transform $f\mapsto \widehat f$ and the natural isomorphism between $L^2(\bbR)$ and $\mathfrak H$, referred to in the previous paragraph. 
Observe that $\calU$ transforms the shift operator $U_\tau$ in $L^2(\bbR)$ into the operator of multiplication by $\ee^{\ii k\tau}$ in $\mathfrak H$.  

It follows that $\calU$ effects the Floquet--Bloch decomposition of $\tau$-periodic (i.e. commuting with $U_\tau$) operators on $L^2(\bbR)$. The precise statement is as follows:

\begin{proposition}\label{prp.b1}
Let $A$ be a bounded self-adjoint operator in $L^2(\bbR)$ that commutes with the shift operator 
$U_\tau$. Then $A$ is decomposable as
\begin{equation}
\calU A\calU^*=\int_{(-\pi/\tau,\pi/\tau)}^\oplus A(k)\, \dd k,
\label{eq:d1}
\end{equation}
with some bounded self-adjoint fiber operators $A(k)$ in $\ell^2(\bbZ)$. 
\end{proposition}
\begin{proof} 
As $A$ commutes with shifts, we find that $\calU A\calU^*$ commutes with the operator of multiplication by $\ee^{\ii k\tau}$ in $\mathfrak H$. By taking linear combinations and limits in strong operator topology, we find that $\calU A\calU^*$ commutes with the operators of multiplication by any bounded function of $k$ in $\mathfrak H$.
The proof is now complete by reference to \cite[Theorem XIII.84]{RS4}. 
\end{proof}

\subsection{The integral of the trace of $A(k)$ over the period}

Let $A$ be as in Proposition~\ref{prp.b1}. We want to recall a calculation for the integral of the trace of $A$ over the period. We will not be very precise with our assumptions, since in the application to periodic Hankel operators below it will be clear that all relevant objects are well-defined. Let us assume that $A$ has  the integral kernel $A(x,y)$, continuous in $x,y\in\bbR$. Furthermore, we will assume that the fiber operators $A(k)$ in \eqref{eq:d1} are trace class. 
\begin{proposition}\label{lma:b2}
Under the above assumptions, the identity
\[
\frac1{\tau}\int_0^\tau A(x,x)\dd x=\frac1{2\pi}\int_{-\pi/\tau}^{\pi/\tau}\Tr A(k)\dd k
\]
holds true.
\end{proposition}
\begin{proof}
For every $k\in(-\pi/\tau,\pi/\tau)$, let us denote by $[A(k)]_{n,m}$ the $(n,m)$'th matrix element of the operator $A(k)$ with respect to the canonical basis in $\ell^2$. From the definition of $\calU$ and the decomposition \eqref{eq:d1} we find that the integral kernel of $A(x,y)$ can be represented as
\[
A(x,y)=
\frac1{2\pi}\sum_{n,m\in\bbZ}\int_{-\pi/\tau}^{\pi/\tau}[A(k)]_{n,m}\ee^{\ii(k+\frac{2\pi}{\tau} n)x}\ee^{-\ii(k+\frac{2\pi}{\tau} m)y}\dd k.
\]
Set $x=y$
\[
A(x,x)=
\frac1{2\pi}\sum_{n,m\in\bbZ}\int_{-\pi/\tau}^{\pi/\tau}[A(k)]_{n,m}\ee^{\ii\frac{2\pi}{\tau}(n-m)x}\dd k
\]
and integrate over $x\in(0,\tau)$. Then all terms with $n\not=m$ in the right-hand side vanish, and we obtain
\[
\frac1{\tau}\int_{0}^\tau A(x,x)\dd x
=
\frac1{2\pi}\sum_{n\in\bbZ}\int_{-\pi/\tau}^{\pi/\tau}[A(k)]_{n,n}\dd k
=
\frac1{2\pi}\int_{-\pi/\tau}^{\pi/\tau}\Tr A(k)\dd k,
\]
as claimed. 
\end{proof}

\subsection{The IDS measure in terms of $\bh(k)$: preliminary result}
Let us discuss the Floquet--Bloch decomposition of smooth periodic Hankel operators $\bH_\omega$. We recall (see \eqref{eq:x3}) that 
\[
\bH_\omega=U_{\tau\omega}\bH U_{\tau\omega}^*, \quad \omega\in\bbT,
\]
where $\bH$ is a Hankel operator in $L^2(\bbR)$ that commutes with $U_\tau$. 

We apply the ``abstract'' considerations of the previous two subsections to $A=\bH$. Let us write the Floquet--Bloch decomposition \eqref{eq:d1} for $\bH$ as 
\begin{equation}
\calU\bH \calU^*=\int_{(-\pi/\tau,\pi/\tau)}^\oplus \bh(k)\, \dd k,
\label{eq:d3}
\end{equation}
with some fiber operators $\bh(k)$ acting in $\ell^2(\bbZ)$. It was proved in \cite[Lemma~2.6]{PuSobolev} that under the smoothness assumption \eqref{eq:d4}, the fiber operators $\bh(k)$ are \emph{trace class}. We state the following preliminary result. 

\begin{theorem}
Let $\bH_\omega$ be a smooth $\tau$-periodic Hankel operator and let $\nu$ be the corresponding IDS measure. For any continuous function $\varphi$ compactly supported in $\bbR\setminus\{0\}$ one has 
\begin{equation}
\int_{-\infty}^\infty \varphi(\lambda)\dd\nu(\lambda)
=
\frac1{\pi}
\int_{0}^{\pi/\tau}\Tr \varphi(\bh(k))\dd k,
\label{eq:d5}
\end{equation}
where $\bh(k)$ are the fiber operators in \eqref{eq:d3}. 
\end{theorem}
\begin{proof}
The operators $\bH_\omega$ (and their integer powers) have continuous integral kernels.  Furthermore, these integral kernels are $\tau$-periodic on the diagonal and 
\[
\bH_\omega(x,x)=\bH(x+\tau\omega,x+\tau\omega). 
\]
It follows that for any $\omega\in\bbT$ and any polynomial $\varphi$ with $\varphi(0)=0$, 
\begin{align*}
\lim_{M\to\infty}\frac1{2M}\Tr(\chi_M\varphi(\bH_\omega)\chi_M)
=
\frac1{\tau}\Tr(\chi_{(0,\tau)}\varphi(\bH)\chi_{(0,\tau)}). 
\end{align*}
Applying Proposition~\ref{lma:b2}, we obtain 
\begin{align*}
\frac1{\tau}\Tr(\chi_{(0,\tau)}\varphi(\bH)\chi_{(0,\tau)})
=
\frac1{2\pi}
\int_{-\pi/\tau}^{\pi/\tau}\Tr \varphi(\bh(k))\dd k.
\end{align*}
 Comparing with formula~\eqref{eq:c28} for the density of states measure, we arrive at
\[
\int_{-\infty}^\infty \varphi(\lambda)\dd\nu(\lambda)
=
\frac1{2\pi}
\int_{-\pi/\tau}^{\pi/\tau}\Tr \varphi(\bh(k))\dd k.
\]
Finally, from the symmetries of $\bh(k)$ (see \eqref{eq:d7} below) it is easy to see that $\bh(-k)$ is unitarily equivalent to $\bh(k)$. This allows to reduce integration over $(-\pi/\tau,\pi/\tau)$ in the right-hand side to the integration over $(0,\pi/\tau)$, and we obtain \eqref{eq:d5}. 
\end{proof}

\subsection{The band functions}

Let $\{E_n(k)\}$ be some enumeration of the eigenvalues of $\bh(k)$. Then \eqref{eq:d5} can be written, somewhat informally,  as
\[
\int_{-\infty}^\infty \varphi(\lambda)\dd\nu(\lambda)
=
\frac1{\pi}
\int_{0}^{\pi/\tau}\sum_n\varphi(E_n(k))\dd k.
\]
Heuristically we can think of this as 
\begin{equation}
\nu=\frac1{\tau}\sum_n m\circ E_n^{-1},
\label{eq:d-n3}
\end{equation}
where $m$ is the normalised Lebesgue measure on $(0,\pi/\tau)$. 
A precise interpretation of \eqref{eq:d-n3} requires some analysis of eigenvalue branches $\{E_n(k)\}$. Indeed, without such analysis it is not even clear if the number of elements in the set $\{E_n(k)\}$ is the same for different $k$. Here we rely on the results of \cite{PuSobolev}. 

\begin{proposition}\cite[Theorem~1.4]{PuSobolev}
\label{thm.branches}
There is a finite or countable list $\mathcal L$ of non-vanishing real-valued real analytic functions $E = E(k)$ on $\bbR$ (we call them \emph{band functions}) that represent all non-zero eigenvalues of $\bh(k)$ in the following sense:
\begin{itemize}
\item
for each $k\in(0,\pi/\tau)$ and each eigenvalue $E_*\not = 0$ of $\bh(k)$  of multiplicity $m\geq1$ there are exactly $m$ functions $E_1, E_2, \dots, E_m \in \mathcal L$ (not necessarily distinct, i.e. a function can be repeated on the list $\mathcal L$ according to multiplicity) such that $E_* = E_1(k) = E_2(k) = \dots = E_m(k)$; 
\item
conversely, for each $k\in(0,\pi/\tau)$ and each $E\in\mathcal L$ there is an eigenvalue of $\bh(k)$ that coincides with $E(k)$. 
\end{itemize}
\end{proposition}
We emphasize that $\mathcal L$ is a \emph{list} rather than a \emph{set}, i.e. the same function can be listed several times in $\mathcal L$, which corresponds to multiplicities of eigenvalues. We write the list $\mathcal L$ as the disjoint union 
\[
\mathcal L={\mathcal L}_{\text{non-flat}}\cup {\mathcal L}_{\text{flat}}
\]
where ${\mathcal L}_{\text{flat}}$ is the set of all band functions that are identically constant (we call them \emph{flat band functions}) and ${\mathcal L}_{\text{non-flat}}$ are the rest of the band functions (we call them \emph{non-flat band functions}). For the flat bands, we have 

\begin{proposition}\cite[Lemma~6.1]{PuSobolev}\label{prp:d-n7}
If $E\in {\mathcal L}_{\text{\rm flat}}$ is a flat band function, then $-E$ is also in ${\mathcal L}_{\text{\rm flat}}$, with the same multiplicity.
\end{proposition}

Next we turn to non-flat band functions. We recall that $\sigma_n$ is our notation for the range of the band function $E_n$. 

\begin{proposition}\cite[Theorem~6.3]{PuSobolev}
\label{prp:d-n8}
\begin{enumerate}[\rm (i)]
\item
if $E_n$ is a non-flat band function, then it appears on the list ${\mathcal L}_{\text{\rm non-flat}}$ exactly once. In other words, all non-flat bands have multiplicity one. 
\item
If $E_n$ is a non-flat band function, then $E_n'(k)\not=0$ for $k\in(0,\pi/\tau)$.
\item
If $\sigma_n^{\rm int}=\{E_n(k): 0<k<\pi/\tau\}$, then for any $n\not=m$, we have $\sigma_n^{\rm int}\cap\sigma_m^{\rm int}=\varnothing$. In other words, the interior parts of non-flat bands are disjoint. 
\item 
If $\sigma_n$ and $\sigma_m$ are two non-flat bands, then $(-\sigma_n)\cap\sigma_m=\varnothing$.
\end{enumerate}
\end{proposition}

Parts (i)--(iii) of this proposition show that we can enumerate the non-flat band functions consistently on the interval $(0,\pi/\tau)$. For example, we can write ${\mathcal L}_{\text{non-flat}}$ as a disjoint union of two lists,
\[
{\mathcal L}_{\text{non-flat}}=\{E_n^+\}\cup\{E_n^-\},
\]
where $\{E_n^+\}$ are positive band functions and $\{E_n^-\}$ are negative band functions. Each of the two lists may be empty, finite or countable. In each of the two lists, we can enumerate the band functions monotonically so that 
\[
E_n^+(k)>E_{n+1}^+(k) \quad\text{ and }\quad E_n^-(k)<E_{n+1}^-(k)
\]
for all indices $n$ and all $k\in(0,\pi/\tau)$. 
The important aspect here is that the non-flat band functions do not intersect each other on $(0,\pi/\tau)$.

We can now give a precise meaning to formula \eqref{eq:d-n3}, which allows us to complete the proof of Theorem~\ref{thm:d-n1}. 

\subsection{Proof of Theorem~\ref{thm:d-n1}}
We write formula \eqref{eq:d-n3} with $\{E_n\}$ being the band functions of Proposition~\ref{thm.branches}. The right-hand side of \eqref{eq:d-n3} splits into the sum over flat band functions and over non-flat band functions. We first note that since $\bh(k)$ is compact, the band functions (whether flat or non-flat) can only accumulate to zero. By the reflection symmetry expressed by Proposition~\ref{prp:d-n7}, the flat band functions produce the sum \eqref{eq:d-n1}. 

By Proposition~\ref{prp:d-n8}(ii), each of the non-flat band functions produces an a.c. measure supported on a closed interval. By Proposition~\ref{thm.branches}, all band functions are non-vanishing, and therefore their ranges do not contain zero. Finally, Proposition~\ref{prp:d-n8}(i), (iii) and (iv) ensures that the band functions satisfy the required intersection properties. The proof of Theorem~\ref{thm:d-n1} is complete. \qed

\section{Finite band periodic Hankel operators}
\label{sec.dd}

\subsection{The structure of the fiber operators $\bh(k)$}
In this section, we give two concrete simple examples of periodic Hankel operators with flat and non-flat bands. In order to do this, we need to explain the structure of the fiber operators $\bh(k)$ in the Floquet--Bloch decomposition \eqref{eq:d3}. Let $\widetilde P_n$ be the Fourier coefficients of $P$ defined in \eqref{eq:d6}. It was proved in \cite{PuSobolev} that the operator $\bh(k)$ in $\ell^2(\bbZ)$ can be identified with the matrix 
\begin{equation}
[\bh(k)]_{n,m}=
B\big(\tfrac12-i\tfrac{2\pi}{\tau} n-ik, \tfrac12+i\tfrac{2\pi}{\tau} m+ik\big)\, \wt P_{n-m}, \quad  n,m\in\bbZ,
\label{eq:d7}
\end{equation}
where $B$ is the Beta function. Using the identity
\[
B(a,b) = \frac{\Gamma(a) \Gamma(b)}{\Gamma(a+b)},
\]
representation \eqref{eq:d7} can be rewritten as follows:
\begin{equation}
\label{eq:d8}
[\bh(k)]_{n,m} = \overline{\gamma_n(k)} \widetilde{\Sigma}_{n-m}\gamma_m(k), 
\end{equation}
where
\begin{align}\label{eq:d9}
\wt{\Sigma}_n = \frac{\wt P_n}{\Gamma(1-\ii\tfrac{2\pi}{\tau} n)}, 
\quad
\gamma_n(k)  = \Gamma\big(\tfrac12 + \ii(\tfrac{2\pi}{\tau} n + k)\big). 
\end{align}
Notation $\wt{\Sigma}_n$ is motivated as follows. Suppose our periodic Hankel operator $\bH$ is positive, and let $\Sigma$ be the measure in the representation \eqref{eq:a4}. Periodicity of $\bH$ means that $\Sigma$ is $\tau$-periodic, and a simple computation shows that $\wt{\Sigma}_n$, defined above in \eqref{eq:d9}, are exactly the Fourier coefficients of $\Sigma$. 

The utility of \eqref{eq:d8} comes from the fact that, at least formally, 
it represents $\bh(k)$ as the product of three operators in $\ell^2(\bbZ)$: multiplication by 
$\gamma_n(k)$, convolution with the sequence $\widetilde{\Sigma}_n$ 
and again multiplication by $\overline{\gamma_n(k)}$. 

\subsection{A single band operator}\label{sec:d5}
Let us consider the particular case when $\Sigma$ is a single point mass periodically continued to the real line with the period $\tau$. More precisely, we set 
\[
\Sigma=\sum_{n=-\infty}^\infty \delta_{\tau n},
\]
where $\delta_x$ is a point mass at $x$. 
Then $\wt{\Sigma}_n=\frac1\tau$ for all $n\in\bbZ$ and 
\[
[\bh(k)]_{n,m} =\frac1\tau
\overline{\gamma_n(k)} \gamma_m(k),
\]
where $\gamma_n(k)$ is as in \eqref{eq:d9}. Clearly, this is a rank one operator. The non-zero eigenvalue $E_0(k)$ of this rank one operator is easy to compute:
\begin{equation}
E_0(k)
=\frac1\tau\sum_{n=-\infty}^\infty\abs{\gamma_n(k)}^2
=\frac1\tau\sum_{n=-\infty}^\infty\pi\sech\pi(\tfrac{2\pi}{\tau} n+k).
\label{eq:d10}
\end{equation}
In fact, up to an appropriate scaling,  $E_0(k)$ coincides with the standard Jacobi elliptic function ${\rm dn}(k)$ (see Appendix~\ref{sec:A3}). The function $E_0(k)$ is even, periodic with period $\tfrac{2\pi}{\tau}$, strictly positive and real-analytic on the real line and extends to the complex plane as a meromorphic function. It is strictly decreasing on $(0,\tfrac{\pi}{\tau})$ with $E_0(k)>0$, with a non-degenerate maximum at $0$ and a non-degenerate minimum at $\tfrac{\pi}{\tau}$. 

Let us come back to the IDS measure $\nu$ in this case. Denote 
\begin{equation}
E_{\max}=E_0(0)=
\frac1\tau\sum_{n=-\infty}^\infty\pi\sech\pi(\tfrac{2\pi}{\tau} n)
\label{eq:d12}
\end{equation}
and 
\begin{equation}
E_{\min}=E_0(\tfrac{\pi}{\tau})=
\frac1\tau\sum_{n=-\infty}^\infty \pi\sech\pi(\tfrac{2\pi}{\tau} (n+\tfrac12)).
\label{eq:d13}
\end{equation}
We have 
\[
0<E_{\min}<E_{\max}.
\]
The measure $\nu$ is absolutely continuous, supported on the interval $[E_{\min},E_{\max}]$ and can be written as
\[
\nu=\frac1\tau m\circ E_0^{-1},
\]
where $m$ is the normalised Lebesgue measure on $(0,\pi/\tau)$. Since the extrema of $E_0$ are non-degenerate, the density of $\nu$ has square-root singularities near $E_{\min}$ and $E_{\max}$.

\subsection{Flat bands}\label{sec:d6}
Let us consider a modification of the previous example, where $\Sigma$ is the signed measure
\[
\Sigma
=\sum_{n=-\infty}^\infty\left( \delta_{\tau n} - \delta_{\frac{\tau}{2}+\tau n}\right).
\]
Then the Fourier coefficients of $\Sigma$ are 
\[
\wt{\Sigma}_n=\frac{1-(-1)^n}{\tau}
\]
and the formula for the fiber operator $\bh(k)$ becomes
\[
[\bh(k)]_{n,m} =\frac1\tau\bigl(\overline{\gamma_n(k)}\gamma_m(k)-(-1)^n\overline{\gamma_n(k)}(-1)^m\gamma_m(k)\bigr).
\]
Thus, $\bh(k)$ is a rank two operator. An elementary calculation shows that the eigenvalues of $\bh(k)$ are $\pm E_*(k)$, where $E_*(k)>0$ satisfies
\begin{equation}
\tau^2E_*(k)^2=
\left(\sum_{n=-\infty}^\infty \sech\pi(\tfrac{2\pi}{\tau} n+k)\right)^2
-
\left(\sum_{n=-\infty}^\infty (-1)^n\sech\pi(\tfrac{2\pi}{\tau} n+k)\right)^2.
\label{eq:d14}
\end{equation}
By an identity for Jacobi elliptic functions (see Lemma~\ref{lma:b3} in Appendix~\ref{sec:A3}), the right-hand side here is constant, $E_*(k)=E_*>0$. 
Thus, both spectral bands are flat. The spectrum of $\bH$ consists of three eigenvalues of infinite multiplicity: $-E_*,0,E_*$. In particular, the density of states measure is pure point, supported at the points $-E_*$ and $E_*$. (By our convention, the eigenvalue $0$ is ``invisible'' to the measure $\nu$.)

The study of flat bands is now an active branch of solid state theory, see e.g. \cite{Le:24} and references therein.

\subsection{The kernels of positive periodic Hankel operators}
\label{sec:d7}
Let $\bH$ be a positive $\tau$-periodic Hankel operator in $L^2(\bbR)$, and let $\Sigma$ be the measure in the representation \eqref{eq:a4}. As already discussed, periodicity of $\bH$ means that $\Sigma$ is $\tau$-periodic.

We have two possible cases:

Case (i): $\Sigma$ on the period $[0,\tau)$ is a finite linear combination of $N<\infty$ point masses. In this case, by Proposition~\ref{prop:b8}(ii), the kernel of $\bH$ is infinite-dimensional. Furthermore, it is easy to see that the convolution with the sequence $\widetilde\Sigma_n$ is a rank $N$ operator. By \eqref{eq:d8} it follows that the rank of the fiber operator $\bh(k)$ is $N$, so there are exactly $N$ band functions and therefore $\nu(\bbR_+)=N/\tau$. This agrees with Theorem~\ref{thm:D1}. 

To summarise: in this case $\Ker \bH$ is infinite-dimensional and $\nu(\bbR_+)<\infty$. 

Case (ii): $\Sigma$ on the period $[0,\tau)$ is \emph{not} a finite linear combination of point masses. Then, by Proposition~\ref{prop:b8}(ii), the kernel of $\bH$ is trivial. 
On the other hand, it is not difficult to see that the fiber operator $\bh(k)$ in this case has infinite rank, so there are infinitely many band functions (accumulating at $\lambda=0$) and therefore $\nu(\bbR_+)=\infty$. Again, this agrees with Theorem~\ref{thm:D1}.

To summarise: in this case $\Ker \bH$ is trivial and $\nu(\bbR_+)=\infty$.

We can summarise this discussion as follows: in the case of positive periodic Hankel operators, in the formula 
\[
\nu(\bbR_+)+\dim\Ker\bH_\omega=\infty
\]
(see \eqref{eq:xx1}), exactly one of the two terms in the left-hand side is infinite.

\section{The random Kronig--Penney--Hankel model}
\label{sec.g} 

\subsection{Preliminaries}
We return to the rKPH model, introduced in Example~\ref{ex:b3}:
\[
\bH_\omega=\sum_{n\in\bbZ}\varkappa_\omega(n) \jap{\cdot,\psi_n}\psi_n
\quad \text{ in $L^2(\bbR)$,}
\]
where $\psi_n(\xi)=\beta(\xi-\tau n)$ with $\beta$ defined in \eqref{eq:a4} and $\{\varkappa_\omega(n)\}_{n\in\bbZ}$ are i.i.d. random variables with the common distribution $\bbP_0$. Throughout this section, we assume that $\supp\bbP_0$ is a compact set on the positive half-line separated away from the origin, and denote 
\begin{equation}
0<\varkappa_{\min}=\inf(\supp\bbP_0)
\quad \text{ and }\quad
\varkappa_{\max}=\sup(\supp\bbP_0)<\infty.
\label{eq:varkappa}
\end{equation}

The vectors $\{\psi_n\}_{n\in\bbZ}$ are normalised by $\norm{\psi_n}^2=1/2$, but they are neither orthogonal nor complete in $L^2(\bbR_+)$. Sums of rank one projections with random coefficients of this type have been studied in our related work \cite{Pa-Pu} in a more abstract setting, and we are going to use the results of \cite{Pa-Pu} here. 

Let us compute the Gram matrix of the vectors $\{\psi_n\}_{n\in\bbZ}$: 
\begin{align}
\jap{\psi_n,\psi_m}
=
\int_{-\infty}^\infty \beta(\xi-\tau n)\beta(\xi-\tau m)\dd \xi=
\frac12\sech\frac{\tau}{2}(n-m).
\label{eq:g8}
\end{align}
Crucial for us is that the inner product $\jap{\psi_n,\psi_m}$ depends only on $\abs{n-m}$ and decays exponentially fast as $\abs{n-m}\to\infty$. In \cite{Pa-Pu} it is shown that under these assumptions, the spectral properties of $\bH_\omega$ can be expressed in terms of the distribution $\bbP_0$ and the $2\pi$-periodic function $\varphi$ termed the \emph{symbol} in \cite{Pa-Pu}: 
\[
\varphi(k)
=\sum_{n=-\infty}^\infty \jap{\psi_n,\psi_0}\ee^{\ii nk}
=\frac12 \sum_{n=-\infty}^\infty\sech\frac{\tau n}{2}\ee^{\ii nk},
\]
$k\in(-\pi,\pi)$. By an application of the Poisson summation formula, we can rewrite $\varphi$ as
\[
\varphi(k)
=
\frac{\pi}{\tau}\sum_{m=-\infty}^\infty \sech\tfrac{\pi}{\tau}(k-2\pi m).
\]
Recalling formula \eqref{eq:d10} for the band function $E_0$ of a single-band periodic Hankel operator, we recognise the above function as 
\[
\varphi(k)=E_0(k/\tau). 
\]

\subsection{The location of spectrum}
We recall the notation $E_{\min}$ and $E_{\max}$ from \eqref{eq:d12}, \eqref{eq:d13}; the range of values of the symbol $\varphi$ is $[E_{\min},E_{\max}]$. 
\begin{theorem}
The deterministic spectrum of the rKPH model $\bH_\omega$ is
\[
\sigma(\bH_\omega)=\supp\bbP_0\cdot[E_{\min},E_{\max}],
\]
where the right-hand side stands for the product set, 
\[
\supp\bbP_0\cdot[E_{\min},E_{\max}]
=
\{ab: a\in \supp\bbP_0, \, b\in [E_{\min},E_{\max}]\}.
\]
\end{theorem}
\begin{proof}
This is a direct consequence of  \cite[Theorem~2.2]{Pa-Pu}, which applies to more general symbols $\varphi$.
\end{proof}

In particular, using the notation \eqref{eq:varkappa} and
\[
\sigma_{\min}=\varkappa_{\min}E_{\min}
\quad \text{ and }\quad
\sigma_{\max}=\varkappa_{\max}E_{\max},
\]
we see that 
\[
\sigma(\bH_\omega)\subset [\sigma_{\min},\sigma_{\max}]
\]
and the endpoints $\sigma_{\min}$ and $\sigma_{\max}$ belong to the spectrum. The spectrum may or may not have gaps, depending on the support of $\bbP_0$.

\subsection{The IDS measure and Lifshitz tails}
We first note that the IDS measure $\nu$ of the rKPH model is finite, with 
\[
\nu(\bbR_+)=1/\tau.
\]
This is a direct consequence of Theorem~\ref{thm:D1}, because the density of $\supp\Sigma_\omega$ here is $1/\tau$.

Next,  we can apply \cite[Theorem~2.5]{Pa-Pu} to obtain the following asymptotic formulas for $\nu$ near the upper and lower endpoints of the spectrum $\sigma_{\max}$ and $\sigma_{\min}$, known as Lifshitz tails:
\[
\nu((\sigma_{\max}-\delta,\sigma_{\max}])\sim \ee^{-C\delta^{-1/2}}
\quad\text{ and }\quad
\nu([\sigma_{\min},\sigma_{\min}+\delta))\sim \ee^{-C\delta^{-1/2}}
\]
as $\delta\to0$. More precisely, we have 
\begin{theorem}\cite[Theorem~2.5]{Pa-Pu}
Assume that the distribution $\bbP_0$ is not supported at one point, i.e. $\varkappa_{\min}<\varkappa_{\max}$. Further, assume that $\bbP_0$ satisfies
\begin{equation}
\bbP_0([\varkappa_{\min},\varkappa_{\min}+\eps))\geq C\eps^\ell
\quad\text{ and }\quad
\bbP_0((\varkappa_{\max}-\eps,\varkappa_{\max}])\geq C\eps^\ell
\label{eq:g2}
\end{equation}
for some $\ell>0$.
Then we have the following Lifshitz tails behaviour of the density of states measure $\nu$ of the rKPH model at the top and bottom endpoints of the spectrum:
\begin{align*}
\lim_{\delta\to0_+} \log \bigl(-\log \nu((\sigma_{\max}-\delta,\sigma_{\max}])\bigr)/\log\delta &= -\frac12\, ,
\\
\lim_{\delta\to0_+} \log\bigl(-\log \nu([\sigma_{\min},\sigma_{\min}+\delta))\bigr)/\log\delta &= -\frac12\, .  
\end{align*}
\end{theorem}
Assumption \eqref{eq:g2} means that $\bbP_0$ is ``not too thin'' near the endpoints of its support. For example, if $\bbP_0(\{\varkappa_{\max}\})>0$ and $\bbP_0(\{\varkappa_{\min}\})>0$, then \eqref{eq:g2} is satisfied. Assumption \eqref{eq:g2} is needed only for the lower bounds in the Lifshitz tails asymptotics.

\subsection{Wegner type estimate}
We use \cite[Theorem~2.4]{Pa-Pu}, which gives the following statement. 
\begin{theorem}\cite[Theorem~2.4]{Pa-Pu}
Assume that $\bbP_0$ is absolutely continuous with the density $\rho$ uniformly bounded by $\rho_{\max}<\infty$. Then the IDS measure $\nu$ of the rKPH model is absolutely continuous with the density satisfying the Wegner type estimate
\[
\frac{\dd\nu(\lambda)}{\dd\lambda}
\leq
\frac{\rho_{\max}\varkappa_{\max}}{\lambda}, \quad \lambda>0.
\]
\end{theorem}
We note that in \cite{Pa-Pu}, we used the normalisation $\norm{\psi_n}=1$, while in this paper we have $\norm{\psi_n}^2=1/2$, but after rescaling the right-hand side of the Wegner estimate does not change.

\subsection{Anderson type localisation}

We note that as $\tau\to\infty$, the inner products $\jap{\psi_n,\psi_m}$ for $n\not=m$ tend to zero, see \eqref{eq:g8}. This asymptotic regime was studied in \cite[Theorem~2.6]{Pa-Pu}. This yields the following result. 

\begin{theorem}
Assume that the distribution $\bbP_0$ is uniformly H\"older continuous on $[\varkappa_{\min},\varkappa_{\max}]$:
\[
\bbP_0([x-\delta,x+\delta])\leq C\delta^\gamma, \quad \forall x\in [\varkappa_{\min},\varkappa_{\max}]
\]
with some $\gamma\in(0,1]$ and $C>0$ independent of $x$. Then there exists $\tau_0>0$ such that for all $\tau\geq\tau_0$, the spectrum of the rKPH model $\bH_\omega$ is pure point almost surely.
\end{theorem}

\section{Conclusion and open questions}

In conclusion, we would like to highlight the key points of our work and raise some questions. 

Ergodic Hankel operators split into two classes: those with trivial kernel and those with infinite-dimensional kernel. Spectral properties of operators in these two classes are quite different. 

Within the class of Hankel operators, positive operators is the most tractable subclass with remarkable properties. This is also the class that displays striking similarities with the theory of the full-line Schr\"odinger operators.

For positive  Hankel operators, we have 
\[
\dim\Ker {\bH}_\omega+\nu(\bbR_+)=\infty. 
\]
In simplest cases (e.g. periodic operators or the rKPH model) exactly one of the two terms on the left-hand side is infinite.  However, there are examples where both terms can be infinite (see Section~\ref{sec:c8}).

For positive operators we have 
\[
\nu(\bbR_+)=D(\supp\Sigma_\omega),
\]
where the $D(\supp\Sigma_\omega)$ is the density of the support of $\Sigma_\omega$, see \eqref{eq:cc6}. 

If $\nu(\bbR_+)=\infty$, the divergence is always due to the singularity at the origin. 
The nature of this singularity invites further analysis; in Theorem~\ref{thm:cc8} we give only the simplest statement in this direction. The singularity of $\nu$ at zero is in some ways analogous to the high-energy asymptotics of the IDS of the full-line Schr\"odinger operator.

The two classes of ergodic Hankel operators where we are able to make considerable progress are periodic Hankel operators and the rKPH model. The study of other classes (such as almost-periodic Hankel operators) presents a promising direction for future work.

In order to keep our exposition as simple as possible, we have only discussed \emph{bounded} Hankel operators in this paper. A theoretical framework exists also for studying unbounded self-adjoint Hankel operators, see e.g. \cite{PuTreil1,Ya4}. Natural examples of unbounded ergodic Hankel operators include the rKPH model where the distribution of $\bbP_0$ is Gaussian.

\appendix

\section{The Carleson condition}

Here we prove that the Carleson condition \eqref{eq:a9} for the measure $\sigma$ is equivalent to the uniform local boundedness condition \eqref{eq:a13} for the measure $\Sigma$. It is clear that \eqref{eq:a13} is equivalent to 
\begin{equation}
\Sigma([n-1,n))\leq C'_\Sigma
\label{eq:A1a}
\end{equation}
with $C'_\Sigma$ independent of $n\in\bbZ$. 

The Carleson condition rewritten in terms of the measure $\Sigma$ reads
\begin{equation}
\int_{-\infty}^A\ee^\xi\dd\Sigma(\xi)\leq C_\sigma e^A, \quad A\in\bbR.
\label{eq:A1}
\end{equation}
Assume \eqref{eq:A1} holds and denote 
$\sigma_n=\Sigma([n-1,n))$. 
We have 
\begin{align*}
\sigma_n
=\int_{[n-1,n)}\dd\Sigma(\xi)
\leq\ee^{-n+1} \int_{[n-1,n)}\ee^\xi\dd\Sigma(\xi)
\leq \ee^{-n+1}\int_{-\infty}^n \ee^\xi\dd\Sigma(\xi)
\leq \ee C_\sigma.
\end{align*}
Conversely, assume \eqref{eq:A1a} and let us prove \eqref{eq:A1}. It is clear that it suffices to prove \eqref{eq:A1} for integer $A$. We have
\begin{align*}
\int_{-\infty}^A \ee^\xi\dd\Sigma(\xi)
= \sum_{n=-\infty}^A \int_{[n-1,n)}\ee^\xi\dd\Sigma(\xi)
\leq \sum_{n=-\infty}^A \ee^n \sigma_n 
\leq C'_\Sigma\sum_{n=-\infty}^A \ee^n
= \frac{C'_\Sigma}{1-\ee^{-1}}\ee^A,
\end{align*}
as required.

\section{An identity for elliptic functions}
\label{sec:A3}

\begin{lemma}\label{lma:b3}
The right-hand side of \eqref{eq:d14} is independent of $k\in\bbR$. 
\end{lemma}
\begin{proof}
In the theory of elliptic functions the letter $k$ is standardly used to denote elliptic modulus (see below). Thus, in order to avoid a notation clash, we will temporarily use $x$ instead of $k$ to denote the variable in \eqref{eq:d14}.

We start by recalling the definitions of the Jacobi elliptic functions $\dn$ and $\cn$. 
These are doubly periodic, meromorphic functions on the complex plane. Their properties are determined by a parameter $k \in [0, 1]$ known as the \emph{elliptic modulus}. The \emph{complementary modulus} is defined by $k' = \sqrt{1-k^2}$. Periodicity is described by the \emph{quarter periods}, $K$ and $K'$, which are given by the complete elliptic integral of the first kind: $K = K(k)$,
\[
K(k)=\int_0^{\pi/2}\frac{\dd\theta}{\sqrt{1-k^2(\sin\theta)^2}}\, ,
\]
and $K' = K(k')$.
The functions $\dn (u,k)$ and $\cn (u,k)$ have the following standard Fourier series representations:
\begin{align} \label{eq:dn_series}
\dn (u, k) &= \frac{\pi}{2K} + \frac{\pi}{K}\sum_{n=1}^{\infty} \sech(\tfrac{n\pi K'}{K}) \cos\left(\tfrac{n\pi u}{K}\right)\, ,
\\ 
\label{eq:cn_series}
    \cn (u, k) &= \frac{\pi}{kK} \sum_{n=0}^\infty \sech\left( \tfrac{(2n+1)\pi K'}{2K} \right) \cos\left(\tfrac{(2n+1)\pi u}{2K}\right)\, .
\end{align}
These functions are related by several fundamental identities, including the following key relation:
\begin{equation} \label{eq:key_identity}
\dn(u, k)^2 - k^2 \cn(u, k)^2 = (k')^2.
\end{equation}

We return to \eqref{eq:d14} and denote 
\[
F(x) = \sum_{n=-\infty}^\infty \sech\pi\left(\tfrac{2\pi}{\tau}n+x\right), \quad 
G(x) = \sum_{n=-\infty}^\infty (-1)^n \sech\pi\left(\tfrac{2\pi}{\tau}n+x\right)\, .
\]
Applying the Poisson summation formula, we obtain the Fourier series for these functions:
\begin{align*} 
F(x) &= \frac{\tau}{2\pi} \left(1 + 2\sum_{n=1}^\infty \sech\left(\tfrac{n\tau}{2}\right) \cos(n\tau x)\right)\, ,
\\ 
G(x) &= \frac{\tau}{\pi} \sum_{n=0}^\infty \sech\left(\tfrac{(2n+1)\tau}{4}\right) \cos\left((n+1/2)\tau x\right)\, .
\end{align*}
We can now identify these series with the elliptic functions defined in \eqref{eq:dn_series} and \eqref{eq:cn_series} by performing a term-by-term comparison. 
Matching the arguments of $\sech$ and $\cos$, we find 
\[
\frac{K'}{K} = \frac{\tau}{2\pi}\quad \text{ and }\quad u = \frac{K\tau x}{\pi}
\]
and therefore 
\[
F(x) = \frac{K\tau}{\pi^2} \dn \left(\tfrac{K\tau x}{\pi}, k\right) 
\quad \text{and} \quad 
G(x) = \frac{kK\tau}{\pi^2} \cn \left(\tfrac{K\tau x}{\pi}, k\right)\, ,
\]
where $k$ is the modulus fixed by the condition $\frac{K'(k)}{K(k)} = \frac{\tau}{2\pi}$. 
From the identity \eqref{eq:key_identity} we obtain 
\[
F(x)^2-G(x)^2=\left(\frac{K\tau k'}{\pi^2}\right)^2,
\]
and so the right-hand side in \eqref{eq:d14} is constant. 
The proof of Lemma~\ref{lma:b3} is complete. 
\end{proof}


\begin{thebibliography}{14}

\bibitem{Aiz-War}
M.~Aizenman, S.~Warzel,
\emph{Random operators,}
Grad. Stud. Math., \textbf{168}
American Mathematical Society, Providence, RI, 2015.

\bibitem{Al-Co:12}
S.~Albeverio, F.~Gesztesy, R.~H{\o}egh-Krohn, H.~Holden,
\emph{Solvable models in quantum mechanics,}
Second edition. 
AMS Chelsea Publishing, Providence, RI, 2005.

\bibitem{CSF}
I.~P.~Cornfeld, S.~V.~Fomin, Ya.~G.~Sina\u{\i}, 
\emph{Ergodic theory,}
Springer-Verlag, New York, 1982. 

\bibitem{Da-Fi:24} 
D.~Damanik, J.~Fillman, 
\emph{One-Dimensional Ergodic Schr\"odinger Operators, II. 
Specific Classes,} American Mathematical Society, 2024.

\bibitem{Fedele}
E.~Fedele, 
\emph{The spectral density of Hankel operators with piecewise continuous symbols,}
Integral Equations Operator Theory \textbf{92} (2020), no. 1, Paper No. 1.

\bibitem{Howland}
J.~Howland, 
\emph{Spectral theory of operators of Hankel type.} I, II.
Indiana Univ. Math. J. \textbf{41} (1992), no. 2, 409--426, 427--434.

\bibitem{Ji:22} 
S.~Jitomirskaya, 
\emph{One-dimensional quasiperiodic operators: global theory, duality, and sharp analysis of small denominators,}
Proc. Int. Cong. Math. Vol. 2, pp. 1090--1120, EMS Press, Berlin, 2022.


\bibitem{Kac}
I.~S.~Kac,
\emph{Spectral multiplicity of a second-order differential operator and expansion in eigenfunction.} (in Russian)
Izv. Akad. Nauk SSSR Ser. Mat. \textbf{27} (1963), 1081--1112.

\bibitem{La-Sa}
A.~Laptev, Yu.~Safarov,
\emph{Szeg\H{o} type limit theorems,}
J. Funct. Anal. \textbf{138} (1996), 544--559.

\bibitem{Le:24}
D.~Leykam, 
\emph{Flat bands, sharp physics,} 
AAPPS Bulletin \textbf{31}:2 (2024).

\bibitem{MPT}
A.~V.~Megretski\u{\i}, V.~V.~Peller, S.~R.~Treil', 
\emph{The inverse spectral problem for self-adjoint Hankel operators,}
Acta Math. \textbf{174} (1995), no.~2, 241--309.


\bibitem{Nikolski}
N.~K.~Nikolski, \emph{Operators, Functions and Systems: An Easy Reading, Volume I: Hardy, Hankel, and Toeplitz,}
Amer. Math. Soc. 2002. 

\bibitem{Pa-Pu}
L.~Pastur, A.~Pushnitski,
\emph{Sums of projections with random coefficients,}
preprint, arXiv:2509.21539


\bibitem{Pa-Fi:92} 
L.~ Pastur, A.~ Figotin,
\emph{Spectra of Random and Almost Periodic Operators}, 
Springer-Verlag, Berlin, 1992.

\bibitem{Peller2}
V.~V.~Peller, 
\emph{Hankel operators in the theory of perturbations of unitary and self-adjoint operators,}
Funct. Anal. Appl. \textbf{19} (1985), 111--123.


\bibitem{Peller}
V.~V.~Peller,
\emph{Hankel operators and their applications,}
Springer 2003.



\bibitem{PuSobolev}
A.~Pushnitski, A.~Sobolev,
\emph{Hankel operators with band spectra and elliptic functions,}
to appear in Duke Math. J., arXiv:2307.09242.

\bibitem{PuTreil1}
A.~Pushnitski, S.~Treil, 
\emph{Unbounded integral Hankel operators,}
to appear in Funct. Anal. and its Applications (D.~Yafaev memorial issue). 
Preprint arXiv:2502.04175.

\bibitem{RS4}
M.~Reed, B.~Simon, 
\emph{Methods of modern mathematical physics. IV. Analysis of operators,}
Academic Press, New York-London, 1978. 

\bibitem{Widder}
D.~V.~Widder, 
\emph{The Laplace Transform,} 
Princeton University Press, Princeton, N. J., 1941.

\bibitem{Widom66}
H.~Widom, \emph{Hankel matrices}. Trans. Amer. Math. Soc. \textbf{121} (1966), 1--35. 

\bibitem{Ya3}
D.~Yafaev,
\emph{Spectral and scattering theory for perturbations of the Carleman operator,}
St. Petersburg Math. J. \textbf{25} (2014), no. 2, 339--359.

\bibitem{yaf_apde_2015} 
D. Yafaev, 
\emph{Criteria for Hankel operators to be sign-definite},
Analysis \& PDE, 
\textbf{8} (2015), No. 1, 183--221.

\bibitem{Ya1}
D.~Yafaev,
\emph{Quasi-diagonalization of Hankel operators,}
Journal d'Analyse Mathematique, \textbf{133} (2017), 133--182.

\bibitem{Ya4}
D.~Yafaev,
\emph{Toeplitz versus Hankel: semibounded operators,}
Opuscula Math. \textbf{38} (2018), no. 4, 573--590.


\end{thebibliography}
\end{document}